\documentclass[12pt]{amsart}

\usepackage{amsmath,amscd,amssymb,amsfonts,enumerate,hhline,color}

\definecolor{hot}{RGB}{65,105,225}

\usepackage[pagebackref=true,colorlinks=true, linkcolor=hot ,  citecolor=hot, urlcolor=hot]{hyperref}

\usepackage[all]{xy}

\def\Supp{{\rm{Supp}\,}}
\def\ord{\text{ord}}
\newcommand{\bA}{{\mathbb A}}
\newcommand{\bC}{{\mathbb C}}

\newcommand{\bN}{{\mathbb N}}
\newcommand{\bP}{{\mathbb P}}

\newcommand{\bQ}{{\mathbb Q}}
\newcommand{\bZ}{{\mathbb Z}}
\def\bT{{\mathbb T}}

\def\bm{{\bf m}}
\def\bn{{\bf n}}
\def\be{{\bf e}}

\def\bt{{\bf t}}
\def\bff{{\bf f}}
\def\bs{{\bf s}}
\def\bsy{\boldsymbol}

\newcommand{\cB}{{\mathcal B}}
\newcommand{\cC}{{\mathcal C}}
\newcommand{\cD}{{\mathcal D}}

\newcommand{\cF}{{\mathcal F}}
\newcommand{\cG}{{\mathcal G}}

\newcommand{\cI}{{\mathcal I}}
\newcommand{\cL}{{\mathcal L}}
\newcommand{\cM}{{\mathcal M}}
\newcommand{\cN}{{\mathcal N}}

\newcommand{\cO}{{\mathcal O}}
\newcommand{\cP}{{\mathcal P}}

\newcommand{\cV}{{\mathcal V}}

\newcommand{\Ima}{{\rm Im\, }}
\newcommand{\rank}{{\rm rank}}
\newcommand{\Hom}{{\rm{Hom}}}

\newcommand{\wti}{\widetilde}

\newcommand{\Spec}{{\rm{Spec\,}}}
\newcommand{\Exp}{{\rm{Exp\,}}}
\newcommand{\codim}{\hbox{\rm codim}\,}

\newcommand{\ra}{\rightarrow}

\newcommand{\id}[1]{\langle  {#1} \rangle}

\def\lra{\longrightarrow}
\def\al{\alpha}

\def\lam{\lambda}

\def\im{{\rm{Im}}}
\def\pa{\partial}
\def\me{\medskip}
\def\bone{\mathbf{1}}
\def\ol{\overline}

\newtheorem{Thm}[]{Theorem}
\newtheorem{Conj}[]{Conjecture}
\newtheorem{Prop}[]{Proposition}
\newtheorem{Cor}[]{Corollary}

\theoremstyle{plain}
\newtheorem{thm}[subsection]{Theorem}
\newtheorem{prop}[subsection]{Proposition}
\newtheorem{lem}[subsection]{Lemma}
\newtheorem{cor}[subsection]{Corollary}
\newtheorem{conj}[subsection]{Conjecture}

\newtheorem{df}[subsection]{Definition}

\theoremstyle{definition}
\newtheorem{rem}[subsection]{Remark}
\newtheorem{exa}[subsection]{Example}
\newtheorem{subs}[subsection]{}

\title[Bernstein-Sato ideals and local systems]{Bernstein-Sato ideals and local systems}
\author{Nero Budur}
\address{KU Leuven, Department of Mathematics,
Celestijnenlaan 200B, B-3001 Leuven, Belgium} 
\email{Nero.Budur@wis.kuleuven.be}
\address{University of Notre Dame, Department of Mathematics, 255 Hurley Hall, IN 46556, USA} \email{nbudur@nd.edu}

\begin{document}

\begin{abstract} The topology of smooth quasi-projective complex varieties is very restrictive. One aspect of this statement is the fact that natural strata of local systems, called cohomology support loci, have a rigid structure: they consist of torsion-translated subtori in a complex torus. We propose and partially confirm a relation between Bernstein-Sato ideals and local systems. This relation gives yet a different point of view on the nature of the structure of cohomology support loci of local systems. The main result is a partial generalization to the case of a collection of polynomials of the theorem of Malgrange and Kashiwara which states that the Bernstein-Sato polynomial of a hypersurface recovers the monodromy eigenvalues of the Milnor fibers of the hypersurface. We also address  a multi-variable version of the Monodromy Conjecture, prove that it follows from the usual single-variable Monodromy Conjecture, and prove it in the case of hyperplane arrangements.
 \end{abstract}

\maketitle

\tableofcontents

\section{Introduction}

\subsection{Bernstein-Sato ideals and local systems.} We first propose a conjectural picture relating Bernstein-Sato ideals with local systems. It is known that the topology of smooth quasi-projective complex varieties is very restrictive. One aspect of this statement is the fact that natural strata of local systems, called cohomology jump loci, have a rigid structure: they consist of torsion-translated subtori in a complex torus, see Budur-Wang \cite{BW}. The structure of cohomology jump loci in various setups is the main theme in previous works such as Green-Lazarsfeld \cite{GL1, GL2}, Arapura \cite{Ar1, Ar2}, Simpson \cite{Si},  Budur \cite{B-ULS}, Libgober \cite{Li}, Dimca-Papadima-Suciu \cite{DPS}, Popa-Schnell \cite{PoS}, Dimca-Papadima \cite{DP}. The union of the  cohomology jump loci forms the cohomology support locus. The conjectural picture we propose gives yet a different point of view on the nature of the structure of cohomology support loci of local systems.

To be more precise, let $F=(f_1,\ldots ,f_r)$ be a collection of non-zero polynomials $f_j$ in $\bC[x_1,\ldots ,x_n]$. The {\it Bernstein-Sato ideal of $F$} is the ideal $B_F$ generated by polynomials $b\in \bC[s_1,\ldots ,s_r]$ such that 
$$
b(s_1,\ldots ,s_r)f_1^{s_1}\cdots f_r^{s_r} =Pf_1^{s_1+1}\cdots f_r^{s_r+1}$$
for some algebraic differential operator $$P\in \bC\left[x_1,\ldots,x_n,\frac{\pa}{\pa x_1},\ldots ,\frac{\pa}{\pa x_n},s_1,\ldots ,s_r\right].$$
The existence of non-zero Bernstein-Sato ideals $B_F$ has been proved by Sabbah \cite{Sab-p}, see also Bahloul \cite{Bah} and Gyoja \cite{G}. In the one-variable case $r=1$, the monic generator of the ideal $B_F$ is the classical Bernstein-Sato polynomial. In general, the ideal $B_F$ and its radical are not always principal, for such examples see Bahloul-Oaku \cite[\S 4.1]{BO}. The ideal $B_F$ is generated by polynomials with coefficients in the subfield of $\bC$ generated by the coefficients of $F$ \cite[\S 4]{BO}.

\begin{Conj}\label{conj1}
The Bernstein-Sato ideal $B_F$ is generated by products of linear polynomials of the form
$$
\al_{1}s_1+\ldots +\al_rs_r+\al
$$
with $\al_j\in\bQ_{\ge 0}$, and $\al\in\bQ_{>0}$.
\end{Conj}

This would imply the same for the radical ideal of $B_F$. The conjecture would refine a result of Sabbah \cite{Sab-p} and Gyoja \cite{G} which states that $B_F$ contains at least one element of this type. In the one-variable case $r=1$, Conjecture \ref{conj1} is due to Kashiwara \cite{Ka}. When $n=2$, every element of $B_F$ is divisible by the linear polynomials defining $(r-1)$-dimensional faces of the jumping polytopes of the local mixed multiplier ideals of $f_1,\ldots, f_r$, by Cassou-Nogu\`es and Libgober \cite[Theorem 4.1]{CL}.

We have originally arrived to conjecture that Bernstein-Sato ideals have this particular shape from computing examples with the library {\tt dmod.lib} in  {\sc Singular}  \cite{Sing, SingDmod}. This paper evolved out of the effort to understand this behavior.  The interpretation of this behavior in terms of cohomology support loci of local systems which we give in this paper is new to our knowledge.

Apart from the positivity statement, Conjecture \ref{conj1} can be seen as the consequence of the following situation, similar to ones occurring frequently in arithmetic geometry and model theory. Let 
$$
\Exp: \bC^r\lra (\bC^*)^r
$$
be the map $x\mapsto \exp(2\pi i x)$. We conjecture that $\Exp$ of the zero locus $V(B_F)$ of the Bernstein-Sato ideal $B_F$ satisfies the conditions of the following result of M. Laurent:

\begin{Thm}\cite{L}\label{ThmE}
Let $Z$ be a Zariski closed subset of $(\bC^*)^r$ defined over  $\ol{\bQ}$ with a Zariski dense subset of torsion points. Then $Z$ is a finite union of torsion translates of complex subtori.
\end{Thm}

Since in all computed examples $\Exp (V(B_F))$ satisfies the assumptions of Theorem \ref{ThmE}, the interesting question is then: where do the torsion translates of subtori come from? Next, we will formulate a  conjecture answering this question. We will then prove this conjecture in one direction and almost prove it in the other direction as well. 

The idea is that one can produce lots of torsion points by restricting to one-parameter subgroups using the classical result of Malgrange and Kashiwara for the hypersurfaces $f_1^{m_1}\ldots f_r^{m_r}$. Then the torsion-translated subtori are obtained by interpolating over all one-parameter restrictions the Milnor monodromies, or equivalently the nearby cycles complexes, of these hypersurfaces. The interpolation is achieved by Sabbah's specialization complex attached to $F=(f_1,\ldots ,f_r)$. We  show that the support of the Sabbah's specialization complex is related with cohomology support loci of rank one local systems.

Let us give more details now. It is important for the rest of the paper to work locally at a point $x$ in $$X:=\bC^n.$$ In this case, we replace in all the above $B_F$ by the local Bernstein-Sato ideal $B_{F,x}$ of the germ of $F$ at $x$ and we also propose the local version of Conjecture \ref{conj1}. It is known that 
$$B_F=\bigcap_{x\in X}B_{F,x},$$
see \cite[Corollary 6]{BO}. Thus, letting $V(I)$ denote the zero locus of an ideal $I$,

$$
V(B_F)=\bigcup_{x\in X}V(B_{F,x}),
$$
so the local version implies the global version of Conjecture \ref{conj1} for the radical ideals. Moreover, this is a finite union since there is a constructible stratification of $X$ such that for $x$ running over a given stratum the Bernstein-Sato ideal at $x$ is constant \cite{BMM}.

The relation with local systems is in two steps. First, we propose a generalization of the well-known result of Kashiwara \cite{Kas} and Malgrange \cite{Mal} which states that the roots of the classical Bernstein-Sato polynomial of a polynomial germ $f$ give the monodromy eigenvalues on the Milnor fiber. In this case, the cohomology of the Milnor fiber is packaged into Deligne's nearby cycles complex $\psi_f\bC_X$. When $r\ge 1$, a generalization of Deligne's nearby cycles functor is the {\it Sabbah specialization functor}
$$\psi_{F}: D^b_c(X,\bC) \ra D^b_c(D, A),$$
where $$D:=\bigcup_{j=1}^rV(f_j)$$ is the union of the zero loci of the $f_j$,
$$A:=\bC[t_1,t_1^{-1},\ldots,t_r,t_r^{-1}],$$
and $D^b_c(.,R)$ is the bounded derived category of constructible sheaves in the analytic topology over a ring $R$.  This functor has been introduced in \cite{Sab}. The action of $A$ on $\psi_{F}\bC_X$ generalizes the monodromy of the Milnor fiber from the case $r=1$. 

For a point $x$ in $D$, denote by 
$$
\Supp_x(\psi_F\bC_X)
$$
the support of $\psi_F\bC_X$ at $x$ as an $A$-module, see Definition \ref{eqSupp}. The ambient space of the support is the torus $(\bC^*)^r$ with affine coordinate ring $A$. For our purposes, we have to take into account the possibility that some $f_j$ do not vanish at $x$, and thus we are lead 
to define the {\it uniform support} 
$$
\Supp_x^{unif}(\psi_F\bC_X) \quad \subset\quad (\bC^*)^r,
$$
see Definition \ref{dfUS}. See \cite[3.8.1]{Ni} for the next result in the $l$-adic setting:

\begin{Thm}\label{ThmGenLib} $\Supp ^{unif}_x (\psi_F\bC_X)$ is a finite union of torsion translated subtori of $(\bC^*)^r$.
\end{Thm}

The following would generalize the classical result of Kashiwara and Malgrange. 

\begin{Conj}\label{conj2}
$$\Exp( V({B}_{F,x})) = \bigcup_{y \in D \text { near }x}\Supp ^{unif}_y (\psi_F\bC_X).$$
\end{Conj}
\noindent The union is taken over points $y\in D$ in a small ball around $x$. However, one can take only the general points $y$ of a fine enough stratification of the singular locus of $D$. Conjecture \ref{conj2} almost implies Conjecture \ref{conj1} for codimension part of the radical of the Bernstein-Sato ideal:

\begin{Prop}\label{PropImp}
Assume Conjecture \ref{conj2}. Let $Z$ be an irreducible component of $V(B_{F,x})$. If $Z$ has codimension 1, then $Z$ is the zero locus of a linear polynomial of the form
$$
\al_{1}s_1+\ldots +\al_rs_r+\al
$$
with $\al_j\in\bQ_{\ge 0}$ and $\al\in\bQ_{>0}$. \end{Prop}

The case of hyperplane arrangements, where the support of the Sabbah specialization complex is a combinatorial invariant, provides a checking ground and evidence for Conjecture \ref{conj2}, see  Corollary \ref{corBFHA} and the Remark thereafter. We show the following partial confirmation of Conjecture \ref{conj2}.

\begin{Thm}\label{thmConjIn}
$$
\Exp( V({B}_{F,x})) \supset \bigcup_{y\in D\text { near }x}\Supp ^{unif}_y (\psi_F\bC_X).
$$
\end{Thm}

 We also make a significant step toward proving the converse of Theorem \ref{thmConjIn}. Let  $\cD_X$ be the sheaf of holomorphic differential operators on $X$. 
  
\begin{Prop}\label{PropCondInd} The converse of Theorem \ref{thmConjIn}, and so the Conjecture \ref{conj2}, holds for any $F$ if the following holds for any $F$ such that the $f_j$ with $f_j(x)=0$  define mutually distinct reduced and irreducible hypersurface germs at $x$: locally at $x$, for all $\al\in V(B_{F,x})$,
$$
\sum_{j=1}^r(s_j-\al_j) \cD_X[s_1,\ldots,s_r]f_1^{s_1}\ldots f_r^{s_r} \not\equiv \cD_X[s_1,\ldots,s_r]f_1^{s_1}\ldots f_r^{s_r}$$
modulo $\cD_X[s_1,\ldots,s_r]f_1^{s_1+1}\ldots f_r^{s_r+1}$.
\end{Prop}

The relation of Bernstein-Sato ideals with local systems is achieved by relating the latter with the Sabbah specialization complex. For a connected finite CW-complex $M$, let $L(M)$ denote the space of complex local systems of rank one on $M$. Then
$$
L(M)=\Hom (H_1(M,\bZ),\bC^*).
$$
Define the  {\it cohomology support locus} (also called the {\it characteristic variety}) of $M$ to be the subset $\cV(M)$ of $L(M)$ consisting of local systems with non-trivial cohomology,
\begin{equation}\label{eqcv}
\cV(M):=\{ \cL \in L(M)\mid H^k(M,\cL)\ne 0 \text{ for some }k \}.
\end{equation}
There are more refined cohomology jump loci of $M$ which can be defined, but we will not be concerned with them in this article. It is known that $\cV(M)$ is a Zariski closed subset of $L(M)$ defined over $\bQ$.

For a point $x$ in $X$, let $U_{F,x}$ be the complement of $D$ in a small open ball centered at $x$,
$$
U_{F,x}:={\rm{Ball}}_x\smallsetminus ({\rm{Ball}}_x\cap D).
$$
There is a natural embedding of $L(U_{F,x})$ into the torus $(\bC^*)^r$ induced by $F$.

\begin{Thm}\label{thm2} If the polynomials $f_j$ with $f_j(x)= 0$ define mutually distinct reduced and irreducible hypersurface germs at $x$, then
$$\Supp _x(\psi_F\bC_X) = \cV(U_{F,x}).$$
\end{Thm}
\noindent There is absolutely no difficulty to understand the relation between $\Supp _x(\psi_F\bC_X)$ and cohomology support loci if the assumptions are dropped, cf. \ref{subsPfTGL}.

We can define the {\it uniform cohomology support locus} with respect to $F$ at $x$, which we denote by $\cV^{unif}(U_{F,x})$, such that it agrees with $\Supp_x^{unif}(\psi_F\bC_X)$ via Theorem \ref{thm2}, see Definition \ref{dfUC}. Hence:

\begin{Thm}\label{thm3} If the polynomials $f_j$ with $f_j(x)= 0$ define mutually distinct reduced and irreducible hypersurface germs at $x$, then
$$\Exp( V({B}_{F,x}))\supset \bigcup_{y\in D\text { near }x}\cV^{unif}(U_{F,y}).$$ Assuming Conjecture \ref{conj2}, equality holds.
\end{Thm}
\noindent Again, there is absolutely no difficulty to understand what happens if the assumptions are dropped.

Let us mention the connection with local Alexander modules. The cohomologies of the stalks of $\psi_{F}\bC_X$ are the multi-variable local homology Alexander modules, as shown by Sabbah \cite{Sab}, see Proposition \ref{propStalk}. In the special case when all the polynomials $f_j$ are homogeneous, the cohomologies of the stalk at the origin of $\psi_{F}\bC_X$ are the multi-variable universal homology Alexander modules introduced by Dimca-Maxim \cite{DM}, see Proposition \ref{propDM}.

\subsection{The geometry of Bernstein-Sato ideals.}\label{subsGBS} Next, information about uniform supports and cohomology support loci leads to  better understanding of the question of what do zero loci of Bernstein-Sato ideals look like. In the case when all $f_j$ are homogeneous polynomials, we give a formula which reduces the computation of uniforms supports to a lower-dimensional, but possibly non-homogeneous case,  see Proposition \ref{propFR}. Hence, conjecturally, the same holds for $\Exp(V(B_F))$. In particular, we obtain:

\begin{Cor}\label{cor1} Let $F=(f_1,\ldots , f_r)$ with $0\ne f_j\in\bC[x_1,\ldots,x_n]$  irreducible and homogeneous of degree $d_j$ defining mutually distinct hypersurfaces with $gcd(d_1,\ldots,d_r)=1$. Let $V$ be the complement in $\bP^{n-1}$ of the union of the zero loci of $f_j$. If $\chi(V)\ne 0$, then
$$
\{d_1s_1+\ldots +d_rs_r+k=0\}\subset V(B_F)
$$
for some $k\in \bZ$.
\end{Cor}

It is tempting to conjecture that $k=n$ in Corollary \ref{cor1}. We do so below for hyperplane arrangements.

In the case of hyperplane arrangements, the homogenous reduction formula can be applied repeatedly to obtain precise combinatorial formulas for the uniform supports of the Sabbah specialization complex, see Proposition \ref{propBFHA}.  Let $F=(f_1,\ldots, f_r)$ be such that $f_j$ are non-zero linear forms defining mutually distinct hyperplanes. The following terminology is defined in Section \ref{subsHA}. For an edge $W$ of the hyperplane arrangement $\prod_{j=1}^rf_j$, let 
$
F_W
$
be the restriction in the sense of hyperplane arrangements of $F$ to $W$. Let 
$$
F_W=\prod_{i=1}^{l_W}F_W^{(i)}
$$
be a total splitting of $F_W$. If we set $F_W^{(i)}=(f_{1,W}^{(i)},\ldots ,f_{r,W}^{(i)})$, let
$$
d_{j,W}^{(i)}:=\deg f_{j,W}^{(i)}.
$$

\begin{Cor}\label{corBFHA}  Let $F=(f_1,\ldots, f_r)$ with $f_j\in\bC[x_1,\ldots,x_n]$ non-zero linear forms defining mutually distinct hyperplanes. Then 
\begin{equation*}
 \bigcup_{W} V(\id{t_1^{d_{1,W}^{(i)}}\ldots t_r^{d_{r,W}^{(i)}}-1\mid i=1, \ldots ,l_W})\quad \subset \quad \Exp(V(B_F)),
\end{equation*}
where the union is over the edges $W$ of the hyperplane arrangement $\prod_{j=1}^rf_j$. Assuming Conjecture \ref{conj2}, equality holds.
\end{Cor}

\noindent {\bf{Remark.}}
This corollary provides support for Conjecture \ref{conj2} in the sense that the conjectured equality in Corollary \ref{corBFHA} can be checked for many particular examples, see Section \ref{secEx}. Note that the left-hand side is completely combinatorial. The conditions  in the Corollary can be relaxed, see Remark \ref{remRel}, however we opted to keep only an esthetically cleaner statement.
\medskip

By specializing $F=(f_1,\ldots,f_r)$ to $\prod_{j=1}^rf_r$ in the above Corollary we obtain the following. Let $f$ be a hyperplane arrangement, $f_W$ the restriction to the edge $W$, and $f_W=\prod_{i=1}^{l_W}f_W^{(i)}$ a total splitting of $f_W$. Let $d_W^{(i)}=\deg f_W^{(i)}$. Denote by $b_f$ the classical one-variable Bernstein-Sato polynomial of $f$, and by $M_{f,x}$ the Milnor fiber of $f$ at $x$. With this notation:

\begin{Cor}\label{corCO} Let $f\in\bC[x_1,\ldots,x_n]$ be a hyperplane arrangement. Then $\Exp(V(b_f))$, which equals the set of all eigenvalues of the monodromy on $H^\bullet(M_{f,x},\bC)$ for $x$ ranging over $f^{-1}(0)$, is a combinatorial invariant. If $f$ is reduced, this is the set
$$
\bigcup_WV(\id{t^{d_W^{(i)}}-1 \mid i=1,\ldots, l_W}),
$$
where the union is over the edges $W$ of $f$.
\end{Cor}

In contrast, U. Walther has announced  that the Bernstein-Sato polynomial $b_f$ of a hyperplane arrangement is not a combinatorial invariant.  A different proof of Corollary \ref{corCO} involving \cite[Theorem 3.1]{L-e} was noticed and communicated to us by A. Libgober.

The following is a multi-variable generalization of  \cite[Conjecture 1.2]{BMT}. This  statement has implications for the Multi-Variable Strong Monodromy Conjecture, see Theorem \ref{thmSM} below.

\begin{Conj}\label{conjND}  Let $F=(f_1,\ldots,f_r)$, where $f_j$ are central hyperplane arrangements in $\bC^n$, not necessarily reduced, of degree $d_j$, and $\prod_{j=1}^rf_j$ is a central essential indecomposable hyperplane arrangement. Then  
$$\{
d_1s_1+\ldots +d_rs_r + n=0
\}\subset V(B_F).$$
\end{Conj}

\subsection{Multi-Variable Monodromy Conjecture.}\label{subsMVMC} We discuss the relation between multi-variable topological zeta functions on one hand, and  Sabbah specialization complexes and  Bernstein-Sato ideals, on other hand.

Let $F=(f_1,\ldots, f_r)$ with $0\ne f_j\in\bC[x_1,\ldots,x_n]$. We keep the notation from 1.1. Let $\mu: Y\ra X$ be a  log resolution of $\prod_jf_j$.  Let $E_i$ for $i\in S$ be the collection of irreducible components of the zero locus of $(\prod_jf_j)\circ\mu$. Let $a_{i,j}$ be the order of vanishing of $f_j$ along $E_i$, and let $k_i$ be the order of vanishing of the determinant of the Jacobian of $\mu$ along $E_i$. For $I\subset S$, let $E_I^\circ=\cap_{i\in I}E_i \smallsetminus\cup_{i\in S\smallsetminus I}E_i$. With this notation, the {\it topological zeta function} of $F=(f_1,\ldots, f_r)$ is

$$
Z_F^{\, top}(s_1,\ldots ,s_r):=\sum_{I\subseteq S}\chi(E_I^{\circ})\cdot\prod_{i\in I}\frac{1}{a_{i,1}s_1+\ldots a_{i,r}s_r+k_i+1}.
$$
This rational function is independent of the choice of log resolution. Define 
$$
PL(Z_F^{\, top}(s_1,\ldots ,s_r))
$$
to be the polar locus in $\bC^r$.

The following is the {\it Topological Multi-Variable Monodromy Conjecture}, slightly different than phrased by Loeser, see \cite{Ni, L-MV}:

\begin{Conj} \label{conjMVMonConj}
$$
\Exp (PL(Z_F^{\, top}))
\subset \bigcup_{x\in D} \Supp^{unif}_x (\psi_F\bC_X).$$
\end{Conj}
When $r=1$, this is the Topological Monodromy Conjecture of Igusa-Denef-Loeser saying that poles of the topological zeta function give eigenvalues of the Milnor monodromy. In fact, in response to a question of V. Shende, the general case follows from the $r=1$ case:

\begin{Thm}\label{thmR1} Let $\cC$ be a class of non-zero polynomials stable under multiplication. If the  Monodromy Conjecture holds for polynomials in $\cC$, then the Multi-Variable Monodromy Conjecture holds for maps $F=(f_1,\ldots,f_r)$ with $f_j$ in $\cC$.
\end{Thm}

Examples of classes of polynomials stable under multiplication for which the Monodromy Conjecture is known include plane curves \cite{Lo} and hyperplane arrangements \cite{BMT}. Thus Theorem \ref{thmR1} reproves the Multi-Variable Monodromy Conjecture  for plane curves due to Nicaise \cite{Ni}, and proves it for hyperplane arrangements.

One can ask how natural is to specialize the Monodromy Conjecture. We define later what it means to specialize $F$ to another collection $G$ of possibly fewer polynomials, see Definition \ref{dfSpP}. For example,  $F=(f_1,\ldots ,f_r)$ specializes to $G=(f_1,\ldots ,f_{r-1})$, and it also specializes to $\prod_{j=1}^rf_j$. In the first example, the specialization loses in some sense $f_r$, where as in the second example none of the $f_j$ are lost. We call the second type a non-degenerate specialization, see Definition \ref{dfSpP}. We show the following naturality with respect to non-degenerate specializations of the Monodromy Conjecture:

\begin{Thm}\label{thmSp} Assume that Conjecture  \ref{conjMVMonConj} holds for a given $F$. If $G$ is a non-degenerate specialization of $F$, then Conjecture \ref{conjMVMonConj} also holds for $G$.
\end{Thm}

Up to now, there has been no multi-variable version  of the Strong Monodromy Conjecture since it was not clear which ideal of Bernstein-Sato type  was the right candidate, see \ref{subsIBST}. Since a strong version should imply the weaker version, the search for the right candidate is related with the search for the multi-variable generalization of the Malgrange-Kashiwara result. Thanks to V. Levandovskyy, we were able access and experiment with implemented algorithms for computing various types of Bernstein-Sato ideals. Based on these computations and based on the other supporting evidence for Conjecture \ref{conj2} put forth in this paper, we make the following {\it Topological Multi-Variable Strong Monodromy Conjecture}:

\begin{Conj}\label{conjTMVSMC}  $$
PL(Z_F^{\, top})\subset V({B}_F) .$$ 
\end{Conj}

Conjecture \ref{conjTMVSMC} implies Conjecture \ref{conjMVMonConj} if we believe Conjecture \ref{conj2}, hence the adjective ``strong".  At the moment we cannot conclude that the $r=1$ case for the Strong Monodromy Conjecture implies the $r\ge 1$ case, nor that the Strong Monodromy Conjecture is compatible with non-degenerate specializations, but see Remark \ref{remSpSMC}. For hyperplane arrangements we reduce Conjecture \ref{conjTMVSMC}  to Conjecture \ref{conjND}, a result which  was proved for $r=1$ in \cite{BMT}:

\begin{Thm}\label{thmSM} If each $f_j$ defines a (possibly-nonreduced) hyperplane arrangement in $\bC^n$ and if  Conjecture \ref{conjND} holds for the restriction $$F_W=(f_{j,W} \mid f_j(W)=0,  j\in\{1,\ldots, r\})$$ of the hyperplane arrangements to any dense edge $W$ of $\prod_{j=1}^rf_j$, then Conjecture \ref{conjTMVSMC} holds for $F$.
\end{Thm}

\medskip
On a different note, there has been recent interest in zeta functions attached to differential forms and possible connections with monodromy-type invariants, see N\'emethi-Veys \cite{NV}. Let
$dx=dx_1\wedge\ldots\wedge dx_n$ and let $\omega$ be an $n$-form on $X$. Define
$$
Z^{\;top,\; \omega}_F(s_1,\ldots ,s_r)
$$
in a similar fashion as $
Z^{\;top}_F(s_1,\ldots ,s_r)
$, but with $k_i$ replaced by $ord_{E_i}\omega$. Note that 
$$
Z^{\;top,\; dx}_F(s_1,\ldots ,s_r) = Z^{\;top}_F(s_1,\ldots ,s_r).
$$
One can ask  {\it what would a Monodromy Conjecture predict for $
Z^{\;top, \;\omega}_F
$ ?} See \cite[1.2]{NV} for a discussion. We propose an answer which is very natural and which says that the Monodromy Conjecture for Forms is a special case of the Multi-Variable Monodromy Conjecture. Clearly
$$
Z^{\;top,\; f_{r}dx}_{(f_1,\ldots ,f_{r-1})}(s_1,\ldots ,s_{r-1}) = Z^{\;top}_{F}(s_1,\ldots ,s_{r-1},1).
$$
Hence, the {\it Topological Multi-Variable Monodromy Conjecture for Forms} should be:
$$
\Exp( PL(Z_{(f_1,\ldots,f_{r-1})}^{\, top,\; f_{r}dx}))
\subset \bigcup_{x\in D} \Supp^{unif}_x (\psi_{F}\bC_X) \cap V(t_{r}-1),$$
and the {\it Topological Multi-Variable Strong Monodromy Conjecture for Forms} should be:
$$
PL(Z_{(f_1,\ldots,f_{r-1})}^{\, top,\; f_{r}dx})
\subset V({B}_{F})\cap V(s_r-1) .$$
The Topological Multi-Variable Monodromy Conjecture for Forms is thus equivalent with the usual single-variable Topological Monodromy Conjecture, by Theorem \ref{thmR1}.

It is a standard procedure to adjust statements involving topological zeta functions to obtain statements involving:  local topological zeta functions, (local) $p$-adic zeta functions, and, more generally, (local) motivic zeta functions. For brevity, we shall skip this discussion.

\subsection{Applications.} One of the main applications of the theory of $\cD$-modules is that it leads to algorithms which can be implemented to compute topological invariants. For example, the classical result of Malgrange and Kashiwara led to algorithms for computing Milnor monodromy eigenvalues via the classical Bernstein-Sato polynomial, the first such algorithm being due to Oaku \cite{Oa}. Similarly, Conjecture \ref{conj2} would provide already-implemented algorithms to compute cohomology support loci of hypersurface germs complements. There are no other known algorithms for cohomology support loci applicable in general. Note that Bernstein-Sato ideals are essential for the current algorithms computing cohomology of local systems on complements of projective hypersurfaces, see Oaku-Takayama \cite{OTa}.

\subsection
{ Acknowledgement.} We would like to thank V. Levandovskyy for help with computing examples and for corrections. For computations of the Bernstein-Sato ideals in this paper we used the library {\tt dmod.lib} in  {\sc Singular}  \cite{Sing, SingDmod}. We would also like to thank F. J. Castro-Jim\'{e}nez, A. Dimca, A. Libgober, L. Maxim, M. Schulze, V. Shende, W. Veys, U. Walther, and Y. Yoon for helpful discussions, and the University of Nice for hospitality during writing part of the article. Special thanks are due to B. Wang who helped correct many mistakes in the original version. This work was partially supported by the National Security Agency grant H98230-11-1-0169 and by  the Simons Foundation grant 245850.

\subsection
{ Notation.} All algebraic varieties are assumed to be over the complex number field. A variety is not assumed to be irreducible. The notation $(f_1,\ldots,f_r)$ stands for a tuple, while $\id{f_1,\ldots,f_r}$ will mean the ideal generated by elements $f_j$ of some ring. Loops and monodromy around divisors are meant counterclockwise, i.e. going once around $\{x=0\}$ in a small loop sends $x^\al$ to $e^{2\pi i\al}x^\al$ for any $\al\in\bC$.

\section{Cohomology support loci}

\subsection{ Local systems and cohomology support loci.} Let $M$ be a connected finite CW-complex of dimension $n$. Let $L(M)$ be the group of rank one complex local systems on $M$. We can identify
$$
L(M)=\Hom(\pi_1(M),\bC^*) = \Hom (H_1(M,\bZ),\bC^*) = H^1(M,\bC^*).
$$
Consider the ring $$B:=\bC[H_1(M,\bZ)].$$ Then $L(M)$ is an affine variety with affine coordinate ring equal to $B$.

\begin{exa}\label{exaB} If $U$ denotes the complement in a small open ball centered at a point $x$ in $\bC^n$ of  $r$ mutually distinct analytically irreducible hypersurface germs, then $H_1(U,\bZ)=\bZ^r$ is generated by the classes of small loops around the branches, and
$L(U)=(\bC^*)^r$, see \cite[(4.1.5)]{Di}. By Libgober \cite{Li}, the cohomology support locus $\cV(U)$ is a finite union of torsion translated subtori of $L(U)$. Subtori are automatically defined over ${\bQ}$.
\end{exa}

\begin{exa}\label{exaV} If $V$ denotes the complement in $\bP^{n-1}$ of $r$ mutually distinct reduced and irreducible hypersurfaces of degrees $d_1,\ldots ,d_r$, then 
$$H_1(V,\bZ)=\left [\bigoplus_{j=1}^r\bZ\cdot \gamma_j\right ]/\ \bZ(d_1\gamma_1+\ldots +d_r\gamma_r),$$
where $\gamma_j$ is the class of a small loop centered at a general point on the $j$-th hypersurface,
see \cite[(4.1.3)]{Di}. By Budur-Wang \cite{BW}, the cohomology support locus $\cV(V)$ of any smooth complex quasi-projective variety $V$ is a finite union of torsion translated subtori of $L(V)$. 
\end{exa}

Let $M^{ab}$ be the universal abelian cover of $M$. In other words, $M^{ab}$ is the cover of $M$ given by the kernel of the natural abelianization map
$$ab: \pi_1(M)\ra H_1(M,\bZ).$$

\begin{df}\label{dfcv} The {\it homological Alexander support } of $M$ is the subset
$$
\wti{\cV}(M)  := \bigcup_k\Supp (H_k(M^{ab},\bC))
$$
of $L(M)$, where $\Supp (H_k(M^{ab},\bC))$ is the support of the $B$-module $H_k(M^{ab},\bC)$ .
\end{df}

The homological Alexander support is almost the same as the cohomology support locus (\ref{eqcv}) from the Introduction, see \cite[Theorem 3.6]{PS}:

\begin{thm}\label{thmChTp} 
$\wti{\cV}(M)$ is the set of local systems of rank one on $M$ with non-trivial homology.
\end{thm}

Since $H^k(M,\cL^{-1})=H_k(M,\cL)^\vee$ for a rank one local system $\cL$, the last result implies:

\begin{cor}\label{corChTp} 
$\cV(M)=\{\cL\mid \cL^{-1}\in \wti{\cV}(M)\}.$
\end{cor}

\begin{exa}\label{exaHCV} $\cV((\bC^*)^r)$ is just a point in $L((\bC^*)^r)=(\bC^*)^r$, corresponding to the  trivial local system.
\end{exa}

\section{Sabbah specialization and local systems}

In this section we write down some properties of the Sabbah specialization complexes. We also prove Theorems \ref{ThmGenLib}, \ref{thm2}, \ref{thm3}, as well as the homogeneous reduction formula mentioned in \ref{subsGBS} and Corollary \ref{cor1}.

 For a ring $R$ and a variety $X$, let $D^b_c(X,R)$ denote the bounded derived category of $R$-constructible sheaves on the underlying analytic variety of $X$.

\subsection{ Sabbah specialization.}\label{subsSSP}
Let $$F=(f_1,\dots , f_r)$$ be a collection of non-zero polynomials $f_j\in\bC[x_1,\ldots ,x_n]$. Let  
\begin{align*}
X&=\bC^n, &\quad &D_j =V(f_j),&
D&=\cup_{j=1}^rD_j, &\quad & U=X\smallsetminus D. 
\end{align*}

 Let $S=\bC^r$, $S^*=(\bC^*)^r$, and denote by $\wti{S^*}$ the universal cover of $S^*$. We denote the affine coordinate ring of $S^*$ by
$$
A=\bC[t_1,t_1^{-1},\ldots,t_r,t_r^{-1}].
$$
Consider the following diagram of fibered squares of natural maps:
$$
\xymatrix{
D \ar@{^(->}[r]_{i_D}& X \ar[d]_{F}& \ar@{_(->}[l]^{j} U \ar[d]_{F_U}  & \ar[l]^{\quad p} \wti{U} \ar[d]_{F_{\wti{U}}} \\
    & S & \ar@{_(->}[l]^{j_S} S^* & \ar[l]^{\quad {p_S}} \wti{S^*}
}
$$

\begin{df}\label{defSab} The Sabbah specialization functor of $F$ is
\begin{align*}
\psi_F=  i_D^* Rj_* Rp_!(j\circ p)^*: D^b_c(X,\bC) & \ra D^b_c(D,A).
\end{align*}
We call ${\psi_F\bC_X}$ the {\it Sabbah specialization complex}. The constructibility $\psi_F$ over $A$ follows from the part (a) of the Lemma \ref{lemSimpl} below.

\end{df}
  
\begin{rem}\label{remOV} (1) This definition is slightly different than \cite[2.2.7]{Sab} where it is called {\it the nearby Alexander complex}. To obtain the definition in {\it loc. cit.}, one has to restrict further to $\cap_jD_j$. 

\me (2)
This definition is also slightly different than the one in \cite{Ni}, where $Rp_!$ is replaced by $Rp_*$. 

\me (3)
When $r=1$, $\psi_f\cF$ as defined here equals $\psi_f\cF[-1]$ as defined by Deligne, see  \cite[p.13]{Br}. 
\end{rem}

A different expression for Sabbah specialization is as follows. Let 
$$\cL=R(p_S)_!\bC_{\wti{S^*}}=(p_S)_!\bC_{\wti{S^*}}.$$  This is the rank-one local system of free $A$-modules on $S^*$ corresponding to the isomorphism $$\pi_1(S^*)\ra \bZ^r=\pi_1(S^*).$$ Define 
$$
\cL^F:= F_U^*\cL .
$$
This is a local system of $A$-modules on $U$. The following is essentially a particular case of  \cite[2.2.8]{Sab} in light of Remark \ref{remOV} (1):

\begin{lem}\label{lemSimpl}
(a) $\psi_F\cF = i_D^*Rj_*(j^*\cF\otimes_{\bC_{U}}\cL^F)$.

(b) In particular, $\psi_F\bC_X = i_D^*Rj_*\cL^F$.
\end{lem}
\begin{proof}
By the projection formula \cite[2.1.3]{Sab}, $Rj_* Rp_!p^*j^*\cF=Rj_*(j^*\cF\otimes_{\bC_U}Rp_!\bC_{\wti{U}})$. On the other hand, we have $Rp_!\bC_{\wti{U}}=Rp_!F_{\wti{U}}^*\bC_{\wti{S^*}}=F_U^*Rp_!\bC_{\wti{S^*}}=\cL^F$.
\end{proof}

\subsection{Multi-variable monodromy zeta function.}  We recall Sabbah's multi-variable generalization of A'Campo's formula for the monodromy zeta function.

\begin{df}\label{eqSupp} For an $A$-module $G$, we denote by
$
\Supp (G)
$
the support of $G$ in $S^*=\Spec A$. For an $A$-constructible sheaf $\cG$ on $X$ and a point $x$ in $X$, the support of $\cG$ at x is the support of the stalk:
$$
\Supp_x(\cG):=\Supp (\cG_x)\subset S^*
$$
The support at the point $x$  of a complex $\cG\in D^b_c(D,A)$ is 
$$
\Supp _x (\cG) := \bigcup _k \Supp _x (H^k(\cG)) \subset S^*.
$$
\end{df}

\begin{df}
The {\it multi-variable monodromy zeta function} of $\cG\in D^b_c(D,A)$ at the point $x$  is defined to be the cycle
$$
\zeta_x (\cG) :=\sum _{P} \chi_x(\cG_P)\cdot V(P),
$$  
where the sum is over prime ideals $P$ of $A$ of  height  one among those such that their zero locus $V(P)\subset \Supp_x(\cG)$, $\cG_P$ is the localization of $\cG$ at the prime ideal $P$, and $\chi_x$ is the stalk Euler characteristic. 
\end{df}

Codimension-one cycles on $S^*$  can be viewed as  rational functions in $t_1,\ldots , t_r$ up to multiplication by a monomial. We will use the rational function notation for the multi-variable monodromy zeta function of Sabbah specialization complex: 
$$
\zeta_x(\psi_{F}\bC_X)(t_1,\ldots ,t_r):=\zeta_x(\psi_{F}\bC_X).
$$

Let $\mu: Y\ra X$ be a  log resolution of $\prod_jf_j$.  Let $E_i$ for $i\in S$ be the collection of irreducible components of the zeros locus of $(\prod_jf_j)\circ\mu$. Let $a_{i,j}$ be the order of vanishing of $f_j$ along $E_i$, and let $k_i$ be the order of vanishing of the determinant of the Jacobian of $\mu$ along $E_i$. For $I\subset S$, let $E_I^\circ=\cap_{i\in I}E_i \smallsetminus \cup_{i\in S-I}E_i$. The following is Sabbah's generalization of A'Campo's formula, see \cite[2.6.2]{Sab}:

\begin{thm}\label{thmSab} If $f_j(x)=0$ for all $j$, then
$$\zeta _x(\psi_{F}\bC_X)(t_1,\ldots ,t_r) = \prod_{i\in I \text{ with } \mu(E_i)=x} (t_1^{a_{i,1}}\ldots t_r^{a_{i,r}}-1)^{-\chi(E_i^\circ)}.$$
\end{thm}

\begin{cor}\label{corZH}
If $f_j$ are non-zero homogeneous polynomials of degree $d_j$ for all $j$, and $\chi(V)\ne 0$, where $V=\bP^{n-1}\smallsetminus \bP(D)$, then
$$
V(t_1^{d_1}\ldots t_r^{d_r}-1)\subset \Supp_0(\psi_F\bC_X).
$$
\end{cor}
\begin{proof}
In this case, one can take a log resolution $\mu:Y\ra X$ which factors through the blow-up at $0$, such that the strict transform $E$ of the blow-up exceptional divisor $E'$ is the only exceptional divisor of $Y$ mapping to $0$, and such that $\mu$ is an isomorphism outside $D$. In this case, $E'=\bP^{n-1}$, $E^\circ=V$. By construction, the log resolution $\mu$ is the blow-up $Y'$ of $X$ at $0$, followed by the extension to $Y'$ of any fixed log resolution of $\bP(D)$ in $E'$, see \cite[1.4]{RV}. This is possible because $Y'$ is a line bundle over $E'$. Then, by Theorem \ref{thmSab},
$$
\zeta_0(\psi_F\bC_X)(t_1,\ldots,t_r)=(t_1^{d_1}\ldots t_r^{d_r}-1)^{-\chi(V)},
$$
and the conclusion follows.
\end{proof}
 
For $r=1$ one recovers a well-known formula for the monodromy zeta function of a homogeneous polynomial, see \cite[p.108]{Di}.

\subsection{ Local Alexander modules.} Let $x$ be a point in $D$ and $i_x:\{x\}\ra D$ the natural inclusion. Let ${\rm{Ball}}_x$ be a small open ball centered at $x$ in $X$, and let $$U_{F,x}={\rm{Ball}}_x\smallsetminus D.$$ 
Let $U_{F,x}^{ab}$ be the universal abelian cover of $U_{F,x}$, and let
$$
B=\bC[H_1(U_{F,x},\bZ)].
$$
Let
$$
\wti{U_{F,x}}:=U_{F,x}\times_{S^*}\wti{S^*}.
$$
Consider the commutative diagram of fibered squares
$$
\xymatrix{
{\rm{Ball}}_x \ar[d]_{F} & \ar@{_(->}[l]U_{F,x} \ar[d]  & \ar[l]_p \wti{U_{F,x}} \ar[d] \\
S  & \ar@{_(->}[l] S^* & \ar[l] \wti{S^*}.
}
$$
 
\begin{df}
The {\it $k$-th local homology Alexander  module} of $F$ at $x$ is the $A$-module
$$
H_{k}(\wti{U_{F,x}},\bC).
$$
The $k$-th local homology Alexander module of $U_{F,x}$ is the $B$-module 
$$
H_{k}(U_{F,x}^{ab},\bC).
$$
\end{df}

The following is essentially a particular case of \cite[2.2.5]{Sab}:

\begin{prop}\label{propStalk} The $k$-th cohomology of the stalk of the Sabbah specialization complex is isomorphic as an $A$-module, up to the switch between of the action of $t_j$ with that of $t_j^{-1}$, to the $k$-th local homology Alexander module of $F$ at $x$:
$$
H^k(i_x^*\psi_{F}\bC_X) 
\cong_A
H_{k}(\wti{U_{F,x}},\bC).
$$
\end{prop}
\begin{proof} By Lemma \ref{lemSimpl}, $
H^k(i_x^*\psi_{F}\bC_X) = (R^kj_*\cL^F)_x
$. The stalk of the sheaf $R^kj_*\cL^F$ at $x$ equals the stalk of the presheaf $V\mapsto H^k(U\cap V,\cL^F)$. Hence 
\begin{equation}\label{eqStalk}
H^k(i_x^*\psi_{F}\bC_X) =H^k(U_{F,x},\cL^F) .
\end{equation}
We have 
\begin{align*}
H_{k}(\wti{U_{F,x}},\bC) &= H^{2n-k}_c(\wti{U_{F,x}},\bC) = H^{2n-k}(Ra_!p_!\bC_{\wti{U_{F,x}}}) \\
&=H^{2n-k}_c(U_{F,x},\cL^F)= H^{2n-k}(Ra_!\cL^F),
\end{align*}
where $a$ is the map to a point. On the other hand, letting $(\cL^F)^\vee$ be the $A$-dual local system of $\cL^F$, and $D_A$ be the Verdier duality functor, we have
\begin{align*}
H^{2n-k}(Ra_!\cL^F) & =H^{n-k}(Ra_!D_A((\cL^F)^\vee[n]))=H^{n-k}(D_ARa_*((\cL^F)^\vee)[n]) \\
& = D_AH^{k}(Ra_*(\cL^F)^\vee) = D_AH^k(U_{F,x},(\cL^F)^\vee).
\end{align*}
The last $A$-module is non-canonically isomorphic, after the change of $t_j$ with $t_j^{-1}$, with $H^k(U_{F,x},\cL^F)$.
\end{proof}

\begin{lem}\label{lemFact} If the polynomials $f_j$ with $f_j(x)=0$ define mutually distinct reduced and irreducible hypersurface germs at $x$, 
then $$\wti{U_{F,x}}=U_{F,x}^{ab}\quad\text{and}\quad
B=A/\id{t_j-1 \mid f_j(x)\ne 0}.
$$
\end{lem}
\begin{proof} 
The second assertion follows from Example \ref{exaB}. For the first assertion, it is enough to show that $\wti{U_{F,x}}$ corresponds to the kernel of the abelianization map
$$
\pi_1(U_{F,x})\ra H_1(U_{F,x},\bZ).
$$
By definition,  $\wti{U_{F,x}}$ is given by the kernel of the composition
$$
\pi_1(U_{F,x})\mathop{\lra}^{F_*}\pi_1(S^* )=\bZ^r\mathop{\lra}^{id}H_1(S^*,\bZ)=\bZ^r. 
$$
Since the codomain is abelian, it is enough to show that the natural direct image
$$
F_*=((f_1)_*,\ldots , (f_r)_*): H_1(U_{F,x},\bZ) \lra H_1(S^*,\bZ)
$$
is injective. By Example \ref{exaB}, $H_1(U_{F,x},\bZ)$ is free abelian generated by the classes of loops $\gamma_j$ centered a general point of ${\rm{Ball}}_x\cap D_j$ for those $j$ such that $f_j(x)= 0$. Let $\delta_j$ be a generator for the first homology of the $j$-th copy of $\bC^*$ in $S^*$. Both assertions of the Lemma follow then from the fact that
$$
f_j(\gamma_{j'})\sim\left\{
\begin{array}{ll}
0 & \text{ if } f_j(x)\ne 0,\text{ or }j\ne j',\\
\delta_j &\text{ if }f_j(x)=0\text{ and }j=j'.
\end{array}
\right.
$$
Indeed, if $f_j(x)\ne 0$, then $\gamma_{j'}$ can be chosen such that $f_j(\gamma_{j'})$ is a loop homologically equivalent to 0. If $j\ne j'$, then $\gamma_{j'}$ can be chosen such that $f_j(\gamma_{j'})$ is a point. If $f_j(x)=0$ and $j=j'$, then $\gamma_j$ can be chosen to intersect at most once every fiber of $f_j$, hence $f_j(\gamma_j)$ is homologically equivalent to $\delta_j$.
\end{proof}

From Proposition \ref{propStalk} and Lemma \ref{lemFact} we obtain:

\begin{cor}\label{corFact}
If the polynomials $f_j$ with $f_j(x)=0$ define mutually distinct reduced and irreducible hypersurface germs at $x$, then the $k$-th cohomology of the stalk of $\psi_F\bC_X$ at $x$ is the $k$-th local homology Alexander module of $U_{F,x}$. More precisely, the action of $A$ on $H^k(i_x^*\psi_{F}\bC_X)$ factors through the action of $B$, and 
$$
H^k(i_x^*\psi_{F}\bC_X)\cong H_{k}(U_{F,x}^{ab},\bC)
$$
as $B$-modules after replacing on the right-hand side the $t_j$-action with the $t_j^{-1}$-action.
\end{cor}

\begin{rem}\label{remSd} 
In the case when one the polynomials $f_r$ is nonsingular outside $\cup_{j=1}^{r-1}D_j$, or, when $f_r$ is a generic linear polynomial through $x$, more information is available about the local Alexander modules from \cite[2.6.3 and 2.6.4]{Sab}.
\end{rem}

\subs\label{subsPfT4}
{\bf Proof of Theorem \ref{thm2}.} In this case, 
$$L(U_{F,x})=\Spec B=\bigcap_{j : f_j(x)\ne 0}V(t_j-1) \subset S^*$$
by Example \ref{exaB}.  Corollary \ref{corFact} implies that $\Supp_x(\psi_F\bC_X)$ equals $\wti{\cV}(U_{F,x})$ via taking reciprocals coordinate-wise due to the change in action of $t_j$ by $t_j^{-1}$. By Corollary \ref{corChTp}, $\wti{\cV}(U_{F,x})$ equals $\cV(U_{F,x})$ via taking reciprocals coordinate-wise. Hence, $\Supp_x(\psi_F\bC_X)=\cV(U_{F,x})$. $\Box$

\subsection{ Uniform support.}\label{subsUS}
Even if the polynomials $f_j$ vanishing at $x$ do not define mutually distinct reduced and irreducible hypersurface germs at $x$,   the proof  of Lemma \ref{lemFact}  together with Proposition \ref{propStalk} show:
\begin{lem}\label{lemRed}
The action of $A$ on $H^k(i_x^*\psi_{F}\bC_X)$ factors through the action of 
$$
A/(\id{t_j-1\mid f_j(x)\ne 0}).
$$
\end{lem}

As a consequence, for a point $x\in D$, the support $\Supp_x(\psi_F\bC_X)$ lies in a sub-torus of $S^*$ of codimension exactly the number of polynomials $f_j$ with $f_j(x)\ne 0$. More precisely, let
$$
\bT_x:=  V(\id{t_j-1\mid f_j(x)\ne 0}) \subset S^*.
$$
Then
$$
\Supp_x(\psi_F\bC_X) \subset \bT_x.
$$
Let  $r_x$ be the codimension of $\bT_x$ in $S^*$, in other words the number of $j$'s with $f_j(x)\ne 0$.

\begin{df} The $F$-natural splitting of $S^*$ at $x$ is the splitting
$$S^*=\bT_x\times (\bC^*)^{r_x}$$
compatible with the splitting
$$
\{j \mid f_j(x)=0\} \cup \{j\mid f_j(x)\ne 0\}.
$$
\end{df}

\begin{df}\label{dfUS}
The {\it uniform support} at $x$ of $\psi_F\bC_X$ is
$$
\Supp^{unif}_x(\psi_F\bC_X) : =(\Supp_x(\psi_F\bC_X))\times (\bC^*)^{r_x} \subset S^*,
$$
the last inclusion being induced by the $F$-natural splitting of $S^*$.  By definition, the uniform support coincides with the usual support when $\bT_x$ is empty, or in other words, when all $f_j$ vanish at $x$.
\end{df}
 
In other words, the uniform support is defined by the same equations as the usual support  except we discard the equations $t_j=1$ for those $j$ such that the hypersurface $f_j$ does not pass through $x$.

Similarly, consider the cohomology support locus $\cV(U_{F,x})$. This is a subvariety of the space of rank one local systems $L(U_{F,x})$, and $$L(U_{F,x})\subset \bT_x,$$  with $\bT_x$ as above. We have an equality $L(U_{F,x})=\bT_x$ if the germs $f_j$ which vanish at $x$ define mutually distinct reduced and irreducible hypersurface germs. 

\begin{df}\label{dfUC}
The {\it uniform cohomology support locus} of $F$ at $x$ is
$$
\cV^{unif}(U_{F,x}) :=  \cV(U_{F,x})\times (\bC^*)^{r_x} \subset S^*,
$$
the last inclusion being induced by the $F$-natural splitting of $S^*$.  By definition, the uniform cohomology support locus coincides with the usual cohomology support locus when $\bT_x$ is empty, or in other words, when all $f_j$ vanish at $x$.
\end{df}

\subsection {Proof of Theorem \ref{thm3}.} It follows from Theorem \ref{thm2} and Theorem \ref{thmConjIn} which we prove latter. $\Box$

\subsection{ Homogeneous polynomials.} 
Assume now that $f_j$ are non-zero homogeneous polynomials for all $j$. We show first that $i_0^*\psi_{F}\bC_X$ recovers the {\it universal Alexander modules} of Dimca-Maxim \cite[\S 5]{DM}. Let $$V=\bP^{n-1}\smallsetminus \bigcup_{j=1}^r\bP(D_j).$$ Let $d_j$ be the degree of $f_j$. Let $V^{ab}$ be the universal abelian cover of $V$. Then $H_k(V^{ab},\bC)$ admits an action of 
$$B:=\bC[H_1(V,\bZ)].$$ If $f_j$ are mutually distinct irreducible homogeneous polynomials, which is the situation considered in \cite{DM}, $$B=A/\id{t_1^{d_1}\cdots t_r^{d_r}-1},$$
and
$$
L(V)=V(t_1^{d_1}\cdots t_r^{d_r}-1)\subset S^*,
$$
see Example \ref{exaV}.

\begin{prop}\label{propDM}
If $f_j$ are  irreducible homogeneous polynomials of degree $d_j$ defining mutually distinct hypersurfaces with $gcd(d_1,\ldots ,d_r)=1$, then the action of $A$ on $i_0^*\psi_{F}\bC_X$ factorizes through $B$, and 
$$
H^k(i_0^*\psi_{F}\bC_X) \cong H_k(V^{ab},\bC)
$$
as $B$-modules after replacing on the right-hand side the $t_j$-action with the $t_j^{-1}$-action.
\end{prop}
\begin{proof} Consider $U_{F,0}$, the complement in a small open ball at the origin of $D$. The natural projectivization map $U_{F,0}\ra V$ has fibers diffeomorphic with $\bC^*$ and is a deformation retract of the restriction of the tautological line bundle of $\bP^{n-1}$ to $V$. Since $gcd(d_1,\ldots ,d_r)=1$, the Picard group of $V$ is trivial. Hence, topologically, 
$$U_{F,0}\approx \bC^*\times V.$$ Fix a section $\sigma:V\ra U_{F,0}$. First, we show that via this section
$$
V^{ab}\approx V\times_{U_{F,0}} U_{F,0}^{ab}.
$$
The cover on the right-hand side is given by the kernel of  the composition
$$
\pi_1(V)\mathop{\lra}^{\sigma_*}\pi_1(U_{F,0})\mathop{\lra}^{ab} H_1(U_{F,0},\bZ).
$$
The cover on the left-hand side is given by the kernel of
$$
\pi_1(V)\mathop{\lra}^{ab} H_1(V,\bZ).
$$
Hence, it is enough to show that the map
$$
\sigma_*:H_1(V,\bZ)\lra H_1(U_{F,0},\bZ)
$$
is injective. By assumption, both groups are free abelian of rank $r-1$ and, respectively, $r$. Hence $\sigma_*$ is compatible with the K\"{u}nneth decomposition
$$
H_1(U_{F,0},\bZ)=H_1(V,\bZ)\oplus \bZ,
$$
by the naturality of the K\"{u}nneth decomposition via cross products, and the injectivity follows. This also shows that $A$ acts on $H_k(V^{ab},\bC)$ via the surjection
$$
A=\bC[H_1(U_{F,0},\bZ)] \lra B=\bC[H_1(V,\bZ)]=A/\id{t_1^{d_1}\ldots t_r^{d_r}-1}
$$
induced by $\sigma_*$. 

Now the Proposition follows from Corollary \ref{corFact} and the fact that $V^{ab}$ is a deformation retract of
$$U_{F,0}^{ab}\approx (\bC^*)^{ab}\times V^{ab}\approx\bC\times V^{ab}.$$
\end{proof}

\begin{prop}\label{propSuppHom}
If $f_j$ are  irreducible homogeneous polynomials of degree $d_j$ defining mutually distinct hypersurfaces with $gcd(d_1,\ldots ,d_r)=1$, then, in $S^*$:
$$\Supp _0(\psi_F\bC_X) =\cV(V) \subset L(V)=V(t_1^{d_1}\cdots t_r^{d_r}-1).$$
The four sets are equal if, in addition, $\chi(V)\ne 0$.
\end{prop}
\begin{proof}
The first equality is new. It follows from Proposition \ref{propDM} and Corollary \ref{corChTp}. If $\chi(V)\ne 0$, the equality follows from Corollary \ref{corZH}.
\end{proof}

\begin{rem}
In the case when one of the homogeneous polynomials is a generic linear form, more information is available about $\cV(V)$, see \cite[Theorem 3.6]{DM} and see also Remark \ref{remSd}.
\end{rem}

For a point $y\in X=\bC^n$ different than the origin, let $[y]$ denote the point in $\bP^{n-1}$ with homogeneous coordinates given by $y$. We denote by
$$
U_{F,[y]}\subset \bP^{n-1}
$$
the complement in a small ball around $[y]$ of $\bP(D)$. Note that if we consider an affine space neighborhood of $[y]$,
$$[y]\in\bA^{n-1}\subset\bP^{n-1},$$
then
$$
U_{F,[y]}=U_{F_{|\bA^{n-1}},[y]}.
$$
There is a homotopy equivalence
$$U_{F,y}\approx_{ht} U_{F,[y]}.$$  
Hence 
$$
\cV(U_{F,y}) =\cV(U_{F,[y]})=\cV(U_{F_{|\bA^{n-1}},[y]}) \subset \bT_y,
$$
with $\bT_y$ as in \ref{subsUS}. Moreover,  
$$
\cV^{unif}(U_{F,y}) =\cV^{unif}(U_{F_{|\bA^{n-1}},[y]})\subset S^*
$$
because the $F$-natural splitting $S^*=T_y\times (\bC^*)^{r_y}$ is the same as the $F_{|\bA^{n-1}}$-natural splitting. We define:
$$
\cV^{unif}(U_{F,[y]}):=\cV^{unif}(U_{F,y}) =\cV^{unif}(U_{F_{|\bA^{n-1}},[y]}).
$$
From this discussion together with Proposition \ref{propSuppHom}, we obtain the following  computation reduction to a lower-dimensional, possibly non-homogeneous, case:

\begin{prop}\label{propFR}  If the polynomials $f_j$ are homogeneous of degree $d_j$ and define mutually distinct reduced and irreducible hypersurfaces, and if $gcd(d_1,\ldots, d_j)=1$, then
$$
\bigcup_{y\in D}\Supp_y^{unif}\psi_F\bC_X= \bigcup_{y\in D}\cV^{unif}(U_{F,y})=\cV(V)\cup\bigcup_{[y]\in\bP^{n-1}}\cV^{unif}(U_{F,[y]}).
$$
If, in addition,  $\chi(V)\ne 0$, then $\cV(V)=L(V)=V(t_1^{d_1}\cdots t_r^{d_r}-1)$. 
\end{prop}

 \subsection{Proof of Corollary \ref{cor1}.} It follows from Theorem \ref{thmConjIn}, which we will prove later, and Proposition \ref{propSuppHom}. $\Box$

\subsection{ Specialization of polynomial maps.}\label{subSpP} Next, we address the question of what happens with the support of $\psi_F\bC_X$ under specialization of the map $F$ in the following sense:

\begin{df}\label{dfSpP} (a) We say that $F=(f_1,\ldots ,f_r)$ {\it specializes to} $G=(g_1,\ldots ,g_{p})$, if $f_j$ and $g_{k}$ are non-zero polynomials in $\bC[x_1,\ldots ,x_n]$, $r\ge p\ge 1$, and $G$ is the composition $F'\circ F$ where 
$$
F'=(f'_1,\ldots ,f'_{p}):\bC^r\lra \bC^{p}
$$
is such that $f'_{k}$ are monomial maps and the induced map on tori $(\bC^*)^r\ra (\bC^*)^{p}$ is surjective. We will also write $$G=F^M$$ where the matrix $M=(m_{kj})$ in $\bN^{p\times r}$ is obtained from writing 
$$
f'_{k} : (\tau_1,\ldots ,\tau_r)\mapsto \tau_1^{m_{k1}}\ldots \tau_r^{m_{kr}}.
$$

(b) We say that $F^M$ is a non-degenerate specialization of $F$ if, in addition, $\sum_{k=1}^p m_{kj}\ne 0$ for all $j$ such that $f_j$ is non-constant.
\end{df}

The condition in (b) guarantees that no non-constant $f_j$ is lost during the specialization.

We will use the notation:
\begin{align*}
S=\bC^r,\quad S_M=\bC^{p},\quad S^*=(\bC^*)^r,\quad S_M^{*}=(\bC^{*})^{p}.
\end{align*}
There is a natural identification of $S^*$ with $L(S^*)$, a tuple of non-zero complex numbers describing the monodromy of a rank one local  system around the coordinate axes. The pull-back of local systems defines a map
$$
\phi_M:S_M^*=L(S_M^*)\lra S^*=L(S^*)
$$
given by
$(\lam_1,\ldots,\lam_p) \mapsto (\lam_1^{m_{11}}\cdots\lam_p^{m_{p1}},\ldots ,\lam_1^{m_{1r}}\cdots\lam_p^{m_{pr}}).$
Denote by
$$A=\bC[t_1,t_1^{-1},\ldots ,t_r,t_{r}^{-1}],$$
$$A_M=\bC[u_1,u_1 ^{-1},\ldots ,u_{p},u_{p}^{-1}]$$
the coordinate rings of $S^*$ and $S^*_M$, respectively. Then $\phi_M$ corresponds to the ring morphism 
$$
\phi^\#_M:A  \lra A_M$$
mapping $t_j$ to $\prod_{k}u_{k}^{m_{kj}}$.

 The following is a consequence of \cite[2.3.8]{Sab}:

\begin{prop}\label{propSuppSp} If $G=F^M$ is a non-degenerate specialization of the polynomial map $F$, then  for all $x$
$$
\phi^{-1}_M(\Supp_x^{unif}(\psi_F\bC_X)) = \Supp_x^{unif}(\psi_G\bC_X).
$$
\end{prop}
\begin{proof} Let $D_M$ be the union of the zero loci of $g_{k}$. The non-degeneracy assumption is equivalent to saying that $D_M=D$. We can assume $x$ is in $D$.

The statement holds for the usual support. Indeed, $\phi^{-1}_M(\Supp_x(\psi_F\bC_X))=\Supp_x (\psi_F\bC_X \mathop{\otimes}^{L}_A A_M)$, and by \cite[2.3.8]{Sab}, 
$
\Supp_x (\psi_F\bC_X \mathop{\otimes}^{L}_A A_M ) = \Supp_x(\psi_G\bC_X).
$

Recall that the uniform support at $x$ is obtained from the $F$-natural splitting $S^*=\bT_x\times (\bC^*)^{r_x}$, where $\bT_x$ is the zero locus $V(\id{t_j-1\mid f_j(x)\ne 0})$ in $S^*$ and contains $\Supp_x(\psi_F\bC_X)$, by setting
$
\Supp_x^{unif} (\psi_F\bC_X) = \Supp_x(\psi_F\bC_X)\times (\bC^*)^{r_x}.
$

Similarly, $S_M^*=\bT_{M,x}\times (\bC^*)^{p_x}$, where $\bT_{M,x}$ is the zero locus $V(\id{u_k-1\mid g_k(x)\ne 0})$ in $S_M^*$, and $\Supp_x^{unif}(\psi_G\bC_X) =\Supp_x(\psi_G\bC_X) \times (\bC^*)^{p_x}$, where the coordinates of the last term $(\bC^*)^{p_x}$ correspond to $u_{k}$ such that $g_{k}(x)\ne 0$. Hence it is enough to show that $\phi_M$, or equivalently $\phi^\#_M$, is compatible with the splittings.

The splitting on $S^*$ is by definition compatible with the splitting 
$$\{1,\ldots, r\}=\{j\mid f_j(x)=0\}\cup\{j\mid f_j(x)\ne 0\}.$$
The splitting on $S_M^*$ is compatible with the splitting
$$
\{1,\ldots, p\}=\{k\mid g_k(x)=0\}\cup\{k\mid g_k(x)\ne 0\}.
$$ 
Hence it is enough to show that these two splittings are compatible under $\phi^\#_M$. More precisely, let $A'=\bC[t_j^{\pm}\mid f_j(x)=0]$ and $A''=\bC[t_j^{\pm}\mid f_j(x)\ne 0]$, so that $A=A'\otimes_\bC A''$. Similarly define $A_M'=\bC[u_k\mid g_k(x)\ne 0]$ and $A_M''=\bC[u_k\mid g_k(x)\ne 0]$, so that $A_M=A_M'\otimes A_M''$. We need to show that the middle horizontal map in the diagram 
$$
\xymatrix{
A' \ar@{-->}[r] \ar[d]& A_M'\ar[d]\\
A=A'\otimes A'' \ar[r]^{\phi_M^\#\quad} \ar@{->>}[d]& A_M=A_M'\otimes A_M''\ar@{->>}[d]\\
A' \ar@{-->}[r]& A_M'
}
$$
induces the other two horizontal maps making the diagram commute, where the top vertical maps are $id\otimes 1$ and the bottom vertical maps are the natural quotient maps given by the $F$ and $F_M$ natural splittings. In other words, we need to show that the set
$$\{k\in\{1,\ldots ,p\}\mid u_{k} \text{ does not appear in }\phi^\#_M(t_j)\text{ for any }j\text{ with }f_j(x)=0\}$$
is the same as the set 
$$
\{ k \in\{1,\ldots ,p\}\mid g_{k}(x)\ne 0 \}.
$$
This is true since both sets equal the set
\begin{align*}
&\{ k \in\{1,\ldots ,p\}\mid m_{kj}=0 \text{ for all } j\text{ with }f_j(x)=0 \}.
\end{align*}
\end{proof}

\begin{exa}\label{exaSpG} $F=(f_1,\ldots ,f_r)$ specializes to $G=\prod_{j=1}^rf_j$ via the $1\times r$ matrix $M=(1 \ldots 1)$. Hence this specialization is non-degenerate. Here $S_M^*\ra S^*$ is the diagonal inclusion and $\phi:A\ra A_M=\bC[u,u^{-1}]$ is given by
$
t_j\mapsto u.
$
In this case, the uniform support of $\psi_G\bC_X$ at $x$ is the same as the usual support:
$$
\Supp_x^{unif}(\psi_G\bC_X) = \Supp_x(\psi_G\bC_X),
$$
and it consists of the eigenvalues of the monodromy on the Milnor fiber of $G$ at $x$ by Remark \ref{remOV}.
\end{exa}

\begin{exa}\label{exaDEL} $F=(f_1,\ldots ,f_r)$ specializes to $G=(f_1,\ldots ,f_{r-1})$. To see this, in Definition \ref{dfSpP} let $F':S\ra S_M=\bC^{r-1}$ be defined by
$
(\tau_1,\ldots ,\tau_r)\mapsto (\tau_1,\ldots ,\tau_{r-1}),
$
that is $M$ is the square identity matrix with the last row deleted. Hence this specialization is degenerate if $f_r$ is non-constant. Here the inclusion $S_M^*\ra S^*$ is given by 
$
(\tau_1,\ldots ,\tau_{r-1})\mapsto (\tau_1,\ldots ,\tau_{r-1},1).
$
The map $A\ra A_M$ is the natural one induced by the quotient $A_M=A/\id{t_r-1}$.
\end{exa}

\begin{lem}\label{lemRedSp} Let $x$ be a point in $X$. For any collection $G=(g_1,\ldots, g_p)$ of hypersurface germs at $x$ in $X$, there exists a collection $F=(f_1,\ldots,f_r)$ of hypersurface germs at $x$ such that the set of $f_j$ with $f_j(x)=0$ define mutually distinct reduced and irreducible germs, and such that $G$ is a non-degenerate specialization of $F$.
\end{lem}
\begin{proof}
Suppose we find $F$ and a matrix $M=(m_{kj})$ such that $G=F^M$. Then we need to ensure the surjectivity of the map $F':S^*\ra S^*_M$ defined by $$(\tau_1,\ldots ,\tau_r)\mapsto (\ldots,\prod_{j=1}^r\tau_k^{m_{kj}},\ldots ).$$
This is achieved if the rank of the matrix $M$ is $p$, since the linear transformation associated to $M$ is the differential of the map $F'$ on the associated Lie algebras of $S^*$ and $S^*_M$.

Let $r$ by the number of analytically irreducible components of the germ $\prod_{k=1}^pg_k$. Let $f_j$ for $1\le j\le r$ be those components. Write 
$$
g_k=\prod_{j=1}^rf_j^{m_{kj}}
$$
and let $M$ be the $p\times r$ matrix $(m_{kj})$. We can assume by permuting the germs $g_k$ that the first $\rank (M)$ rows of $M$ are linearly independent.

Let $F=(f_1,\ldots, f_r,1,\ldots, 1)$, where the number of $1$'s added after $f_r$ is $p-\rank (M)$. Then $F$ satisfies the conditions of the lemma and $G$ is the specialization of $F$ via the rank $p$ matrix of size $p\times (p+r-\rank (M))$
$$
\left[ M\,
\begin{array}{|c}
O \\  
\hhline{|-} I 
\end{array}\right],
$$
where $O$ is the zero matrix, and $I$ is the identity square matrix of size $p-\rank (M)$.
\end{proof}

\subsection{Proof of Theorem \ref{ThmGenLib}}\label{subsPfTGL} By construction of the uniform support, it is enough to prove the statement for $\Supp_x(\psi_F\bC_X)$. By Lemma \ref{lemRedSp}, the germ of $F$ at $x$ in $X$ is the non-degenerate specialization $G^M$ of a map germ $G=(g_1,\ldots, g_p)$ via some matrix $M$, where the set of $g_k$  with $g_k(x)=0$ define mutually distinct reduced and irreducible germs. Thus $\Supp_x(\psi_F\bC_X)$ equals $\phi_M^{-1}(\Supp_x(\psi_G\bC_X))$ by specialization of the support.  By Theorem \ref{thm2}, $\Supp_x(\psi_G\bC_X)$ is the cohomology support locus of $U_{G,x}$, and thus it is a finite union of torsion translated subtori of $L(U_{G,x})$, see Example \ref{exaB}.  Since $\phi_M$ is a torus homomorphism, the same is true for $\Supp_x(\psi_F\bC_X)$.   $\Box$

\subsection{ Thom-Sebastiani.} Next, we state a multiplicative Thom-Sebastiani type of result for the support of the Sabbah specialization complex. First, we have:

\begin{lem}\label{lemTSdec} For $i=1, 2$, let $X_i=\bC^{n_i}$, $F_i=(f_{i1},\ldots ,f_{ir_i})$ with $f_{ij}\ne 0$, $D_i=\cup_jf^{-1}_{ij}(0)$, $x_i\in D_i$.  Then, in $(\bC^*)^{r_1}\times (\bC^*)^{r_2}$, we have
$$
\Supp_{(x_1,x_2)}^{unif}(\psi_{F_1\times F_2}\bC_{X_1\times X_2})=\Supp_{x_1}^{unif}(\psi_{F_1}\bC_{X_1}) \times \Supp_{x_2}^{unif}(\psi_{F_2}\bC_{X_2}).
$$
\end{lem} 
\begin{proof}
Using the notation of Lemma \ref{lemSimpl} adapted to our situation, there is an equality of local systems of $A_1\otimes A_2$-modules
$$
\cL^{F_1\times F_2}=\cL^{F_1}\boxtimes\cL^{F_2},
$$
where $\boxtimes$ denotes the external direct product on $U_1\times U_2$, with $U_i=X_i\smallsetminus D_i$, and $A_i$ is the coordinate ring of $S_i^*=(\bC^*)^{r_i}$. Using then Lemma \ref{lemSimpl} and standard arguments, one can show that
$$
\psi_{F_1\times F_2}\bC_{X_1\times X_2}=\psi_{F_1}\bC_{X_1}\mathop{\boxtimes}^L\psi_{F_2}\bC_{X_2},
$$
Hence the claim holds for the usual supports. See also \cite[Proposition 3.1]{PS} for the same statement for the cohomology support loci. The claim for the uniform support follows easily from Definition \ref{dfUS}.\end{proof}

The following is the multiplicative Thom-Sebastiani property for the support of the Sabbah specialization complex:

\begin{prop}\label{propMTS} With notation as in Lemma \ref{lemTSdec}, let $r=r_1=r_2$. Let $G$ be the map
$$G=F_1\cdot F_2: X_1\times X_2\lra (\bC^*)^r$$
defined by $$(x_1,x_2)\mapsto (f_{11}(x_1)f_{21}(x_2),\ldots ,f_{1r}(x_1)f_{2r}(x_2))$$
for $x_i\in X_i$, $i=1,2$. Then
$$
\Supp_{(x_1,x_2)}^{unif}(\psi_{G}\bC_{X_1\times X_2}) =\bigcap_{i=1,2}\Supp_{x_i}^{unif}(\psi_{F_i}\bC_{X_i}).
$$
\end{prop}
\begin{proof}
Let $S=\bC^r$ and let $F':S\times S\ra S$ be defined by multiplication coordinate-wise. Then $G=F'\circ (F_1\times F_2)$ and thus $F_1\times F_2$ specializes to $G$, cf. Definition \ref{dfSpP}. Hence, by Proposition \ref{propSuppSp}, $\Supp_{(x_1,x_2)}^{unif}(\psi_{F_1\times F_2}\bC_{X_1\times X_2})$ specializes to $\Supp_{(x_1,x_2)}^{unif}(\psi_{G}\bC_{X_1\times X_2})$ via intersection with the diagonal in $S^*\times S^*$. The claim then follows from Lemma \ref{lemTSdec}.
\end{proof}

\section{Bernstein-Sato ideals}

In this section we develop  some properties of Bernstein-Sato ideals. We use them to prove geometrically a weaker version of Theorem \ref{thmConjIn}. With a similar proof, we then prove Theorems \ref{thmR1} and \ref{thmSp}.

\subsection{Ideals of Bernstein-Sato type.}\label{subsIBST} There are ways to define ideals of Bernstein-Sato type different than presented in the Introduction. These other ideals are useful for understanding the Bernstein-Sato ideal $B_F$. So we start with a more general definition.

 Let $X=\bC^n$. Let $D_X$ denote the Weyl algebra of algebraic differential operators on $X$,
$$
D_X=\bC\left[x_1,\ldots,x_n,\frac{\pa}{\pa x_1},\ldots ,\frac{\pa}{\pa x_r}\right]
$$
with the usual commutation relations.

Let $F=(f_1,\ldots , f_r)$ with $0\ne f_j\in\bC[x_1,\ldots, x_n]$. Let $$M=\{ \bm_k \in \bN^r\mid k=1,\ldots , p \}$$ be a collection of  vectors, which we also view as an $p\times r$ matrix $M=(m_{kj})$ with $m_{kj}=(\bm_{k})_j$, with $r, p\ge 1$.

\begin{df}
The Bernstein-Sato ideal associated to $F$ and $M$ is the ideal
$$
B_{F}^{\,M}=B_F^{\bm_1,\ldots,\bm_p }\subset \bC[s_1,\ldots ,s_r]
$$
of all polynomials $b(s_1,\ldots,s_r)$ such that
$$
b(s_1,\ldots,s_r)\prod_{j=1}^rf_j^{s_j} = \sum_{k=1}^p P_k\prod_{j=1}^rf_j^{s_j+m_{kj}}
$$
for some algebraic differential operators $P_k$ in $D_X[s_1,\ldots ,s_r]$. 
\end{df}

\begin{rem} (a)   $B_F$, as defined before, is $B^{\,\bone}_{F}$, where $\bone=(1,\ldots ,1)$.

(b) For a point $x$ in $X$, the {\it local Bernstein-Sato ideal} $B^{\,M}_{F,x}$ is similarly defined by replacing $D_X$ with the ring $\cD_{X,x}$ of germs of holomorphic differential operators at $x$. Then
 $$B_F^{\,M}=\bigcap_{x\in X}B_{F,x}^{\,M},$$
see \cite[Corollary 6]{BO}.

(c) The ideals $B^{\,M}_{F,x}$ are non-zero by Sabbah \cite{Sab-p}, see also  \cite{Bah}, \cite{G}.
\end{rem}

\begin{exa}\label{exaMon} If $f_j$ are monomials, write $f_j=\prod_{i=1}^nx_i^{a_{i,j}}$. Let 
$l_i(s_1,\ldots ,s_r)=\sum_{j=1}^ra_{i,j}s_j.$
Let  $a_i=\sum_{j=1}^ra_{i,j}$.
Then 
$$
B_F=\id{\prod_{i=1}^n \left(l_i(s)+1 \right)\cdots\left(l_i(s)+a_i \right)}.
$$
\end{exa}

\begin{lem}\label{lemCors}\cite[Lemma 1.5]{G} By the correspondence $s_j\leftrightarrow-\pa_{t_j}t_j$ and $f_j\leftrightarrow\delta(t_j-f_j)$, where $\delta(u)$ denotes the Dirac delta function, i.e. the standard generator of $D_{\bA^1}/D_{\bA^1}u$ with $u$ the affine coordinate on $\bA^1$, there is an isomorphism of $\cD_X[s_1,\ldots,s_r]$-modules
$$\cD_X[s_1,\ldots,s_r]\prod_{j=1}^rf_j^{s_r}\cong \cD_X[-\pa_{t_1}t_1,\ldots ,-\pa_{t_r}t_r]\prod_{j=1}^r\delta(t_j-f_j).$$
The action of $t_j$ on the right-hand side corresponds to replacing $s_j$ by $s_j+1$ on left-hand side.
\end{lem}

 Let $Y=X\times\bC^r$ with affine coordinates $x_1,\ldots,x_n,t_1,\ldots,t_r$. Define for  $\bm\in \bN^r$,
$$
V^\bm D_Y := D_X\otimes_\bC \mathop{\sum_{\boldsymbol{\beta},\bsy{\gamma}\in\bN^r}}_{\bsy{\beta}-\bsy\gamma\ge \bm} \bC t_1^{\beta_1}\ldots t_r^{\beta_r} \pa_{t_1}^{\gamma_1}\ldots \pa_{t_r}^{\gamma_r}\quad \subset\quad D_Y.
$$
The following is the $\cD$-module theoretic interpretation of Bernstein-Sato ideals and it is a consequence of Lemma \ref{lemCors}: 

\begin{prop}\label{corDI} Let $\bm\in \bN^r$. The Bernstein-Sato ideal $B_{F}^{\bm}$ consists of the polynomials $b(s_1,\ldots,s_r)$ such that
$$
b(-\pa_{t_1}t_1,\ldots ,-\pa_{t_r}t_r)\cdot V^{\bsy{0}} D_Y\cdot \prod_{j=1}^r\delta(t_j-f_j)\subset V^\bm D_Y\cdot\prod_{j=1}^r\delta(t_j-f_j).
$$
\end{prop}

The next result unveils somewhat the structure of the Bernstein-Sato ideals. This can be used in practice to compute $B_F$ in cases where the current implementations do not work. It also explains to some extent the nature of the mysterious shifts which appear in Bernstein-Sato ideals.

\begin{thm}\label{thmBSDec} Let $\bm\in\bN^r$. For $j=1,\ldots , r$, let $t_j$ be the ring isomorphism of $\bC[s_1,\ldots ,s_r]$ defined by $t_j(s_i)=s_i+\delta_{ij}$. Then there are inclusions of ideals in $\bC[s_1,\ldots,s_r]$
$$
\prod_{\substack{1\le j\le r\\ m_j>0}}\prod_{k=0}^{m_j-1} t_1^{m_1}\ldots t_{j-1}^{m_{j-1}}t_j^k \,B_F^{\be_j}\subset B_F^\bm \subset \bigcap_{\substack{1\le j\le r\\ m_j>0}}\bigcap_{k=0}^{m_j-1} t_1^{m_1}\ldots t_{j-1}^{m_{j-1}}t_j^k \, B_F^{\be_j}.
$$
\end{thm} 
Here $\delta_{ij}=0$ if $i\ne j$, and $\delta_{ii}=1$. Also, we denote by $\be_j$ the $r$-tuple with the $k$-th entry $\delta_{jk}$. By convention, $t_j^0$ is the identity map, the product map $t_1^{a_1}\ldots t_r^{a_r}$ means the obvious composition of maps, and $t_1^{a_1}\ldots t_r^{a_r} I$ is the image of the ideal $I$ under this product map.

The first inclusion in Theorem \ref{thmBSDec} can be strict, see Examples \ref{exaFI} and \ref{exaNonP}. We do not know examples for which the second inclusion is strict, raising the obvious question if equality holds in general.

The radical of product of ideals equals the radical of the intersection of the ideal. Thus, letting $V(I)$ denote the zero locus of an ideal $I\subset\bC[s_1,\ldots,s_r]$ in $\bC^n$, and writing $t_j$ also for the corresponding linear map of $\bC^n$, we obtain:

\begin{prop}\label{propVBDec} With the notation as above,
$$V(B_F^\bm)=\bigcup _{\substack{1\le j\le r\\ m_j>0}}\bigcup_{k=0}^{m_j-1} t_1^{m_1}\ldots t_{j-1}^{m_{j-1}}t_j^k\, V(B_F^{\be_j}).$$
\end{prop}

Theorem \ref{thmBSDec} is a consequence of the following result, in which we will use the notation $\bt^\bm=\prod_{j=1}^rt_j^{m_j}$.

\begin{lem}\label{lemDec} Let $\bm, \bn\in \bN^r$. Then there are inclusions of ideals
$$
B_F^\bm\cdot (\bt^\bm\, B_F^\bn)\subset B_F^{\bm+\bn}\subset B_F^\bm \cap (\bt^\bm\, B_F^\bn).
$$
\end{lem}
\begin{proof} We will use the notation $\bff^\bs=\prod_{j=1}^rf_j^{s_j}$.
Let  $b_1\in B_F^\bm$ and $b_2\in B_F^\bn$. Write $b_1\bff^\bs=P_1\bff^{\bs+\bm}$ and $b_2\bff^\bs=P_2\bff^{\bs+\bn}$ for some $P_1$ and $P_2$ in $D_X[s_1,\ldots,s_r]$. Apply $\bt^\bm$ to both sides of $b_2\bff^\bs=P_2\bff^{\bs+\bn}$. We obtain then that $(\bt^\bm b_2)\bff^{\bs+\bm}=(\bt^\bm P_2)\bff^{\bs+\bm+\bn}$. Applying $P_1$ on the left on both sides of the equality, we have
$$
P_1(\bt^\bm P_2)\bff^{\bs+\bm+\bn}=P_1(\bt^\bm b_2)\bff^{\bs+\bm}=(\bt^\bm b_2)P_1\bff^{\bs+\bm}=(\bt^\bm b_2)b_1\bff^{\bs}.
$$
Thus $b_1(\bt^\bm b_2)$ is in $B_F^{\bm+\bn}$, which implies the first inclusion.

Take now $b\in B_F^{\bm+\bn}$. Write $b\bff^\bs=P\bff^{\bs+\bm+\bn}$ for some $P$ in $D_X[s_1,\ldots,s_r]$. Then $b\bff^\bs=P\bff^\bn\bff^{\bs+\bm}$, so  $b\in B_F^\bm$. Now, multiply by $\bff^\bm$ on the left on both sides of the last equality. We obtain that $b\bff^{\bs+\bm}=\bff^mP\bff^{\bs+\bm+\bn}$. This shows that $b\in \bt^\bm B_F^\bn$. Hence $b\in B_F^\bm\cap (\bt^\bm B_F^\bn)$, which proves the second inclusion.
\end{proof}

\noindent
{\it Proof of Theorem \ref{thmBSDec}.} Write $\bm=(m_1,0,\ldots, 0)+(0,m_2,\ldots,m_r)$ and apply Lemma \ref{lemDec} to obtain that $B_F^\bm$ is squeezed between the product and the intersection of the ideals $B_F^{m_1\be_1}$ and  $\bt^{m_1\be_1}B_F^{(0,m_2,\ldots,m_r)}$. Repeat the procedure with  $B_F^{(0,m_2,\ldots,m_r)}$ and so on. It remains to squeeze the ideals $B_F^{m_j\be_j}$, for which there is the obvious procedure: write $m_j\be_j=(m_j-1)\be_j+\be_j$, use Lemma \ref{lemDec}, and repeat. $\Box$

\begin{rem}
Note that one can choose different decompositions of $\bm$ than the ones used in the proof of Theorem \ref{thmBSDec}. In particular, if $\pi$ is any permutation of $\{1,\ldots ,r\}$, there are inclusions
$$
\cP_{F,\bm,\pi}:=\prod_{\substack{1\le j\le r\\ m_{\pi(j)}>0}}\prod_{k=0}^{m_{\pi(j)}-1} t_{\pi(1)}^{m_{\pi(1)}}\ldots t_{\pi(j-1)}^{m_{\pi(j-1)}}t_{\pi(j)}^k \,B_F^{\be_{\pi(j)}}\subset B_F^\bm$$
and
$$
B_F^\bm\subset \bigcap_{\substack{1\le j\le r\\ m_{\pi(j)}>0}}\bigcap_{k=0}^{m_{\pi(j)}-1} t_{\pi(1)}^{m_{\pi(1)}}\ldots t_{\pi(j-1)}^{m_{\pi(j-1)}}t_{\pi(j)}^k \,B_F^{\be_{\pi(j)}}=:\cI_{F,\bm,\pi}.
$$
Again, we do not know an example where $B_F^\bm\ne\cI_{F,\bm,\pi}$. The ideals $\cP_{F,\bm,\pi}$ could be different for different permutations $\pi$, see Example \ref{exaNonP}. Hence we the record the following as a strengthening of Theorem \ref{thmBSDec} and Proposition \ref{propVBDec}.
\end{rem}

\begin{prop}\label{propSTR}
With the notation as above, there is an inclusion of ideals
$$
\sum_{\pi}\cP_{F,\bm,\pi}\subset B_F^\bm \subset \bigcap_\pi \cI_{F,\bm,\pi}
$$
where $\pi$ ranges over all permutations of $\{1,\ldots, r\}$. Hence, there is an equality of zero loci
$$
V(B_F^\bm)=\bigcup_{\substack{1\le j\le r\\ m_{\pi(j)}>0}}\bigcup_{k=0}^{m_{\pi(j)}-1} t_{\pi(1)}^{m_{\pi(1)}}\ldots t_{\pi(j-1)}^{m_{\pi(j-1)}}t_{\pi(j)}^k \,V(B_F^{\be_{\pi(j)}}).
$$
\end{prop}

\begin{exa}
Let $F=(f)$ for a polynomial $f$. This is the case $r=1$. Let $\bm=(m)$ with $m\in\bN$. Then $B_F^\bm=\id{b_{f^m}(s/m)}$, where $b_{f^m}(s)$ is the classical Bernstein-Sato polynomial of $f^m$. Thus Theorem \ref{thmBSDec} states in this case that $\prod_{j=0}^{m-1}b_f(s+j)$ is divisible by $b_{f^m}(s/m)$ which, in turn, is divisible by the lowest common multiple $lcm\{b_f(s+j)\mid j=0,\ldots ,m-1\}$.
\end{exa}

\begin{exa}\label{exaFI} In the previous example, let $f=x^4-y^2z^2$. Then $$b_f(s)=(s+1)^3(4s+3)^2(4s+5)^2(2s+3)(4s+7).$$ Thus,
$$
lcm\{b_f(s),b_f(s+1)\}=\frac{b_f(s)b_f(s+1)}{4s+7}.
$$
One also computes that the right-hand side equals $b_{f^2}(s/2)$. Hence the first inclusion in Theorem \ref{thmBSDec} is strict in this case.
\end{exa}

We write next a few immediate consequences of Theorem \ref{thmBSDec} and
Proposition \ref{propVBDec}.

\begin{lem}\label{lemStr}
 Let $\bm\in\bN^r$. Then there is an equality 
$$\Exp(V(B_F^\bm))= \bigcup_{j=1}^r \Exp(V(B^{m_j\be_j}_F)).$$ 
In particular, there is a decomposition
$$\Exp(V(B_F))= \bigcup_{j=1}^r \Exp(V(B^{\be_j}_F)).
$$
\end{lem}
\begin{exa}
The last lemma does not necessarily hold without exponentiating. For example, let $F=(xy, (1-x)y)$. Then one can compute with {\tt dmod.lib} \cite{SingDmod} that $$B_F=\id{(s_1+1)(s_2+1)(s_1+s_2+1)(s_1+s_2+2)},$$ but $B_F^{\be_1}=\id{(s_1+1)(s_1+s_2+1)}$ and $B_F^{\be_2}=\id{(s_2+1)(s_1+s_2+1)}$. Hence the component $V(s_1+s_2+2)$ of $V(B_F)$ does not show up in either $V(B_F^{\be_j})$.

\end{exa}

\begin{lem}\label{lemm} If $m\in \bN$ is nonzero, then 
$$
\Exp (V(B_F^{m\be_j}))=\Exp (V(B_F^{\be_j}))
$$
for $1\le j \le r$.
\end{lem}

\begin{exa} The last lemma does not necessarily hold without exponentiating. Let $r=1$, $F=(x)$, and $m=2$. Then $B_F^{\be_1}=B_F=\id{s+1}$, but $B^{2\be_1}_F=\id{(s+1)(s+2)}$.
\end{exa}

\begin{lem}\label{lemExpEq}  If $\bm=(m_j), \bm'=(m'_j) \in\bN^r$ such that $m_j=0$ if and only if $m'_j=0$ for all $1\le j\le r$, then
$$\Exp (V(B_F^\bm))=\Exp (V(B_F^{\bm'}))=\bigcup_{j:m_j\ne 0}\Exp (V(B_F^{\be_j})).$$
\end{lem}

The following will not be used, but we state it for clarification purposes and it follows from the definition.

\begin{lem}
For a matrix $M\in \bN^{p\times r}$ with row vectors $\bm_k$ for $1\le k\le p$,
\begin{align*}
B_{F}^{\,M} \supset \sum_{k=1}^p B^{\bm_k}_{F},\quad V(B_{F}^{\,M})\subset\bigcap_{k=1}^p V(B^{\bm_k}_{F}).\end{align*}
\end{lem}

All statements in this subsection are true for local Bernstein-Sato ideals as well.

\begin{exa}\label{exaNonP} Let $F=(z,x^4+y^4+2zx^2y^2)$. This appeared in  \cite[Example 3]{BO}. One computes with {\tt dmod.lib} \cite{SingDmod}:
\begin{align*}
B_F&=\langle (s_1+1)(s_2+1)^2(2s_2+1)(4s_2+3)(4s_2+5)(2s_2+3)\rangle,\\
B_{F,0}&=\langle (s_1+1)(s_2+1)^2(2s_2+1)(4s_2+3)(4s_2+5)(s_1+2),\\
 &\quad\quad (s_1+1)(s_2+1)^2(2s_2+1)(4s_2+3)(4s_2+5)(2s_2+3)\rangle ,\\
 B_{F}^{\be_1}&=\id{s_1+1}\cap \id{2s_2+1,s_1+2},\\
B_{F}^{\be_2}&=\id{(s_2+1)\prod_{k=2}^6 (4s_2+k)},\\
B_{F}^{\be_1,\be_2}&=\id{s_1+1,(s_2+1)^2}\cap\id{s_1+2,2s_2+1}\cap\bigcap_{k=2,3,5}\id{s_1+1,4s_2+k}.\\
\end{align*}
Theorem \ref{thmBSDec} and also its strengthening, Proposition \ref{propSTR}, imply that
\begin{align*}
B_F^{\be_1}\cdot (t_1B_F^{\be_2})\subset B_F \subset B_F^{\be_1}\cap (t_1B_F^{\be_2}),\\
(t_2B_F^{\be_1})\cdot B_F^{\be_2}\subset B_F \subset (t_2B_F^{\be_1})\cap B_F^{\be_2},
\end{align*}
which can be checked easily from the above formulas. In this example the three ideals $B_F^{\be_1}\cdot (t_1B_F^{\be_2})$, $(t_2B_F^{\be_1})\cdot B_F^{\be_2}$, and $B_F$ are mutually distinct, and the two inclusions in Proposition \ref{propSTR} are equalities.
\end{exa}

\subsection{Bernstein-Sato ideals and specialization.}

With $F$ as above, let $F^M$ be a specialization of $F$ via a $p\times r$ matrix $M=(m_{kj})$ as in \ref{subSpP}, possibly without the surjectivity assumption of that definition. With the notation as in \ref{subSpP}, the pull-back of local systems and the exponential maps from the tangent spaces at the trivial local systems define a commutative diagram

\begin{equation}\label{eqMaps}
\xymatrix{
\wti{S_M^*}:=T_{\bone}L(S_M^*) \approx\bC^p \ar[d]_{\Exp} \ar[r]^{\wti{\phi}_M} &  \wti{S^*}=T_{\bone}L(S^*) \approx\bC^r \ar[d]_{\Exp}\\
S_M^*=L(S_M^*) \ar[r]^{\phi_M} & S^*=L(S^*)
}
\end{equation}
The top right-hand side space was denoted  $\widetilde{S^*}$ in \ref{subsSSP} and we keep the notation. However, we keep in mind that $T_{\bone}L(S^*)$  is the natural ambient space of the Bernstein-Sato ideal $B_F$. The bottom horizontal map is 
$$\phi_M: (\lam_1,\ldots,\lam_p) \mapsto (\lam_1^{m_{11}}\cdots\lam_p^{m_{p1}},\ldots ,\lam_1^{m_{1r}}\cdots\lam_p^{m_{pr}}).$$
The top map is the linear map given by multiplication  on the left by the matrix $M$:
$$
\wti{\phi}_M: (\al_1,\ldots, \al_p)\mapsto ({m_{11}}\al_1+\cdots+{m_{p1}}\al_p,\ldots ,{m_{1r}}\al_1+\cdots+{m_{pr}}\al_p).
$$
Let $\bC[s_1,\ldots,s_r]$ and $\bC[v_1,\ldots,v_p]$ be the coordinate rings on $\wti{S^*}$ and $\wti{S^*_M}$, respectively. The induced map on coordinate rings is
$$
\wti{\phi}_M^\#: s_j\mapsto \sum_{k=1}^pm_{kj}v_k.
$$
The following is straight-forward from the definition:

\begin{lem}\label{lemwphi} For all non-zero $M\in \bN^{p\times r}$ and all vectors $\bm\in \bN^p$,
$$
\wti{\phi}_M^\#(B_F^{\bm\cdot M})\subset B_{F^M}^\bm.
$$
Hence 
$$
\wti{\phi}_M^{-1} (V(B_F^{\bm\cdot M}))\supset V(B_{F^M}^\bm).$$
\end{lem}

\begin{exa}
The converse does not necessarily hold. Let $F=(y^2-x^3,x^5)$, $\bm=\bone$, $M=[1\quad 0]$. Then $B_F^{\bm\cdot M}=B_F^{\be_1}=\id{(s_1+1)\prod_{k=5, 7, 9}(6s_1+10s_2+k)}.$ The map $\wti{\phi}_M^\#$ sends $s_1\mapsto v$ and $s_2\mapsto 0$. Hence $\wti{\phi}_M^\#(B_F^{\bm\cdot M})=\id{(v+1)\prod_{k=5, 7, 9}(6v+k)}$. However, $F^M=(y^2-x^3)$ and $B_{F^M}^\bm=B_{F^M}=\id{(v+1)(6v+5)(6v+7)}$.
\end{exa}

\begin{prop}\label{propBres} For all non-zero $M\in\bN^{p\times r}$ , 
$$
{\phi}^{-1}_M (\Exp (V(B_F)))\supset \Exp (V(B_{F^M})).
$$
\end{prop}
\begin{proof}
We apply the last lemma with $\bm$ the unit vector in $\bN^p$ to obtain that $\wti{\phi}^{-1}_M (V(B_F^{\bone\cdot M}))\supset V(B_{F^M})$. Hence ${\phi}^{-1}_M (\Exp(V(B_F^{\bone\cdot M})))\supset \Exp(V(B_{F^M}))$. Now we apply Lemma \ref{lemExpEq}  to show that $${\phi}^{-1}_M (\Exp (V(B_F^{\bone\cdot M})))\subset {\phi}^{-1}_M (\Exp(V(B_F))).$$
\end{proof}

\begin{conj}
For all non-zero $M\in\bN^{p\times r}$  and all vectors $\bm\in \bN^p$,
$$
{\phi}_M^{-1} (\Exp (V(B_F^{\bm\cdot M}))) = \Exp (V(B_{F^M}^\bm)).
$$
In particular, for all $M$ with nonzero columns, 
$$
\phi^{-1}_M (\Exp(V(B_F))) = \Exp (V(B_{F^M})).
$$
\end{conj}
\begin{rem} For this conjecture to hold it suffices that
$$
\phi_M^{-1}(\Exp (V(B_F^{\be_j})))\subset \Exp (V(B_{F^M}^{\be_k}))
$$
for all $1\le j\le r$, $1\le k\le p$ such that $m_{kj}\ne 0$, according to Lemma \ref{lemExpEq}.
\end{rem}

\begin{exa}\label{exaE} Consider $F=(f,f)$, with $f=x^2+y^3$, and $M=(2\ 2)$, so that $F^M=(f^4)$.
Then,  one computes using \cite{SingDmod} that
$$
B_F=\id{\prod_{k=5,6,7,11,12,13}(6s_1+6s_2+k)}
$$
and 
$$
B_{F^M}=\id{\prod_{k=5,6,7,11,12,13,17,18,19,23,24,25}(24s+k)}.
$$
The last ideal is generated by the classical Bernstein-Sato polynomial of $f^4$. Thus
$$
\Exp (V(B_F))=V(\prod_{k=1,5,6}(t_1t_2-e^{-2\pi i \frac{k}{6}})).
$$
Since the map $\phi_M$ is $\lam\mapsto(\lam^2,\lam^2)$, we see that
$$
\phi^{-1}_M(\Exp (V(B_F)))=\{e^{-2\pi i\frac{k}{24}}\mid k=1,5,6,7,11,12,13,17,18,19,23,24\},
$$
which is the same as $\Exp (V(B_{F^M}))$. Note that $\wti{\phi}_M^{-1}(V(B_F))\ne V(B_{F^M})$ since the inclusion $\supset$ does not hold.

\end{exa}

All statements in this subsection are true for local Bernstein-Sato ideals as well.

\subsection{ Geometric proof of a weaker version of Theorem \ref{thmConjIn}.}\label{subsPT} We shall give now a proof of the statement that one gets by replacing in Theorem \ref{thmConjIn} the set $\Exp(V(B_{F,x}))$ with its analytic Zariski closure, namely that
$$
(\Exp(V(B_{F,x})))^{cl}\supset \bigcup_{y\in D\text{ near }x}\Supp_y^{unif}(\psi_F\bC_X).
$$
 The method   is to use specialization of polynomial maps to reduce the statement to the case $r=1$ for which equality is  known by  Malgrange and Kashiwara. 

We define a subset of $S^*$ by
$$S^o=\bigcup_{\bm\in\bN^r} \Ima (\phi_\bm) ,$$
where $\phi_\bm:S^*_\bm\ra S^*$ is the map $\lam\mapsto (\lam^{m_1},\ldots,\lam^{m_r})$ from (\ref{eqMaps}). Note that a torsion point in $S^*$ must lie in $S^o$. The Zariski closure, analytic or algebraic, of $S^o$ is $S^*$.

By Theorem \ref{ThmGenLib}, the irreducible components of $\Supp^{unif}_y\psi_F\bC_X$ are torsion translated subtori of $S^*$.  By \cite[Proposition 3.3.6]{BG}, for each component the torsion points are Zariski dense. Hence the algebraic and analytic Zariski closure of $S^o\cap \Supp^{unif}_y\psi_F\bC_X$ equal $\Supp^{unif}_y\psi_F\bC_X$.

On the other hand,
\begin{align*}\label{eqSdf}
& S^o\cap \Supp_y^{unif}(\psi_{F}\bC_X) =\bigcup_{\bm\in \bN^r}  \Ima (\phi_\bm)\cap\Supp_y^{unif}(\psi_{F}\bC_X) 
\end{align*}
Consider the specialization $F^\bm$ of $F$ via the vector $\bm$. Then $F^\bm=\prod_{j=1}^rf_j^{m_j}$. By Proposition \ref{propSuppSp},  
$$\Ima (\phi_\bm)\cap\Supp_y^{unif}(\psi_{F}\bC_X)=\phi_\bm(\Supp_y^{unif}(\psi_{F^\bm}\bC_X))$$
if the specialization is non-degenerate, that is, if no coordinate of $\bm$ is zero. Therefore, for such $\bm$, by the classical result of Malgrange and Kashiwara we have
$$
\bigcup_{y\in D\text{ near }x}\phi_\bm(\Supp_y^{unif}(\psi_{F^\bm}\bC_X))= \phi_\bm(\Exp(V(B_{F^\bm ,x}))),
$$
where $B_{F^\bm}$ is generated by the classical Bernstein-Sato polynomial of $F^\bm$. By Proposition \ref{propBres}, this set is included in $\Ima(\phi_\bm)\cap\Exp(V(B_F))$. Hence 
$$
\bigcup_{y\in D\text{ near }x} S^o\cap \Supp_y^{unif}(\psi_{F}\bC_X) \subset 
S^o\cap\Exp(V(B_{F,x}))
$$
away from the zero locus $V(\prod_{j=1}^r(t_j-1))$. Passing to the analytic Zariski closures, we obtain the claim since  $V(\prod_{j=1}^r(t_j-1))$  lies anyway in both $\Exp(V(B_{F,x})$ and $\bigcup_{y\in D\text{ near }x}\Supp_y^{unif}(\psi_F\bC_X)$, being contributed by the smooth points $y$ of the germs $D_j$ near $x$. $\Box$

\subsection{Proof of Theorem \ref{thmR1}.} We follow the same strategy as in \ref{subsPT}. We view the multi-variable topological zeta function $Z^{\, top}_F$ as a rational function on the same ambient space as for the Bernstein-Sato ideal $B_F$, namely on $T_\bone L(S^*)=\wti{S^*}\approx \bC^r$. Similarly for a matrix $M\in\bN^{p\times r}$,  $Z^{\, top}_{F^M}$ is a rational function on $T_\bone L(S_M^*)=\wti{S_M^*}\approx \bC^p$.  The induced map $\wti{\phi}_M^\#$ on coordinate rings  partially extends to one on the function fields, which we also denote $\wti{\phi}_M^\#$, and which is defined for rational functions with polar locus not containing the image vector subspace $\im(\wti{\phi}_M)$.
In particular,
$$
Z^{\, top}_{F^M} = \wti{\phi}_M^\#({Z^{\,top}_F})$$
if $F^M$ is a non-degenerate specialization of $F$, see Definition \ref{dfSpP}. In this case, for polar loci we have
$$
PL(Z^{\, top}_{F^M})\subset \wti{\phi}_M^{-1}(PL(Z^{\,top}_F))\subset PL(Z^{\, top}_{F^M})\cup W_M,
$$
where $W_M$ is the union of the linear codimension-one subvarieties of $\wti{S}^*_M$ which lie in both $\wti{\phi}_M^{-1}(PL(Z^{\,top}_F))$ and $\wti{\phi}_M^{-1}(ZL(Z^{\,top}_F))$, where $ZL$ stands for the zero locus of a rational function. Hence
\begin{equation}\label{eqBzz}
\Exp(PL(Z^{\, top}_{F^M})) \subset \Exp(\wti{\phi}_M^{-1}(PL(Z^{\,top}_F)))\subset {\phi}_M^{-1}(\Exp(PL(Z^{\,top}_F))),
\end{equation}
with the last inclusion an equality if the map $\phi_M$ is injective. Also in this case, 
$$
\phi_M^{-1}(\Exp (PL(Z^{\, top}_F)))\subset \Exp (PL(Z^{\, top}_{F^M})) \cup\Exp (W_M).
$$

The map $\phi_M$ is injective if for example if $M$ consists of only one row $\bm$ with the greatest common divisor of the entries equal to 1. Let us cover $S^*$ discretely by such images. More precisely, let $S^o$ be as in \ref{subsPT}. Note  it is enough to consider the union of $\im (\phi_\bm)$ over  $\bm\in\bN^r$ with  the greatest common divisor of the entries equal to 1. Since we want to consider only non-degenerate specializations via $\bm$, we shall furthermore restrict in the definition of $S^o$ to vectors $\bm$ with non-zero entries. Note that the Zariski algebraic closure of $S^o$ is still $S^*$ after all these restrictions. 

Since $\Exp (PL(Z^{\, top}_F))$ is algebraically Zariski closed, it is the algebraic Zariski closure of $S^o\cap \Exp (PL(Z^{\, top}_F))$. Also,
\begin{align*}
S^o\cap \Exp (PL(Z^{\, top}_F)) &\subset \bigcup_{\bm}  \phi_\bm (\Exp( (PL(Z^{\, top}_{F^\bm})))) \cup \bigcup_{\bm}\phi_\bm(\Exp (W_\bm)).
\end{align*}
By assumption, the Monodromy Conjecture holds for $F^\bm$. Thus there is an inclusion 
$$
\bigcup_{\bm}  \phi_\bm (\Exp( (PL(Z^{\, top}_{F^\bm}))))\subset \bigcup_{\bm} \bigcup_x \phi_\bm(\Supp_x^{unif}(\psi_{F^\bm}\bC_X)),
$$
which is in turn included in
$$
\bigcup_x S^o\cap \Supp_x^{unif}(\psi_{F}\bC_X)
$$
by specialization of supports. Taking closure, we obtain that $\Exp (PL(Z^{\, top}_F))$ is included in the union of $
\bigcup_x  \Supp_x^{unif}(\psi_{F}\bC_X)$ with the algebraic Zariski closure of $\bigcup_{\bm}\phi_\bm(\Exp (W_\bm))$. It remains to show that we can ignore this last term $\bigcup_{\bm}\phi_\bm(\Exp (W_\bm))$.

Note that $W_\bm$ has dimension zero and that $\bigcup_{\bm}\phi_\bm(\Exp (W_\bm))$ equals $S^o\cap \Exp (PL(Z_F^{\, top})\cap ZL(Z_F^{\, top}))$. Let $T$ be an irreducible component of $PL(Z_F^{\, top})$. Write $T\cap ZL(Z_F^{\, top})=Z_{lin}\cup Z_{nonlin}$, where $Z_{lin}$ is the union of linear irreducible components and $Z_{nonlin}$ is the union of all the other components. Since $Z_{lin}\subsetneq T$ and $\Exp(Z_{lin})$ is algebraically Zariski closed, running the previous argument with $S^o$ replaced by $S^o=\cup_\bm \im(\phi_\bm)$ with $\im(\phi_\bm)\cap \Exp (Z_{lin})=\emptyset$, we obtain that $\Exp (T)$ is included in the union of $
\bigcup_x  \Supp_x^{unif}(\psi_{F}\bC_X)$ with the algebraic Zariski closure of $\bigcup_{\bm}\phi_\bm(\Exp (W_\bm)$ where the union is over those $\bm$ such that $\phi_\bm(\Exp (W_\bm))\in \Exp(Z_{nonlin})$. So we can assume that $Z_{lin}$ is empty and it remains to deal with the non-linear term $Z_{nonlin}$.

The problem with $Z_{nonlin}$ is that $\Exp (Z_{nonlin})$ might have the same algebraic Zariski closure as $\Exp (T)$. However, if $Z_{nonlin}\ne \emptyset$, we change slightly the argument: instead of filling the component $\Exp(T)$ discretely with points of type $\im(\phi_\bm)$, we fill it discretely with restrictions of higher dimensional subtori $\Exp(T)\cap \im(\phi_M)$. More precisely, we let now $S^o$ be the union of $\im (\phi_M)$ over matrices of natural numbers $M$ of size $p\times r$ with $1<p<r$ of rank $p$ and such that $\phi_M$ is injective and such that $M$ has no non-zero columns. This ensures in particular that $M$ gives a non-degenerate specialization. Note that the non-emptiness of $Z_{nonlin}$ implies that $r>2$ and $T\cap W_M=\emptyset$. Running the previous argument with this $S^o$, namely using the inductive assumption that the Monodromy Conjecture holds for $F^M$, using the  specialization of supports, and taking algebraic Zariski closure, we obtain that $\Exp (T)$ is included in $
\bigcup_x  \Supp_x^{unif}(\psi_{F}\bC_X). \quad\Box
$

\subsection{ Proof of Theorem \ref{thmSp}.}  Let $F$ specialize to $G$ via the matrix $M$, that is $G=F^M$. The assumption the Topological Monodromy Conjecture holds for $F$ implies that
$$
\wti{\phi}_M^{-1}(\Exp(PL(Z^{\,top}_F)))\subset \bigcup_{x}\phi_M^{-1}(\Supp_x^{unif}(\psi_F\bC_X)).
$$
By Proposition \ref{propSuppSp}, the right-hand side equals
$$
\bigcup_{x}\Supp_x^{unif}(\psi_G\bC_X).
$$
Hence, by (\ref{eqBzz}), the Topological Monodromy Conjecture  holds for $G$ as well. Note that for (\ref{eqBzz}) we need to assume that $M$ gives a non-degenerate specialization. $\Box$

\begin{rem}\label{remSpSMC} 
The analogs of Theorems \ref{thmR1} and \ref{thmSp} for the Strong Monodromy Conjecture, while expected to hold, do not follow with a similar proof. The reason is that the inclusion $V(B_{F^M})\subset\wti{\phi}_M^{-1}(V(B_F))$ can fail, see Example \ref{exaE}.  Note that compatibility with specializations would  place bounds on $PL(Z^{\,top}_F)$, namely it would have to be included in 
$$
\bigcap_{M}\wti{\phi}_M(V(B_{F^M})) \cap V(B_F). 
$$
When $r=1$, this means that $PL(Z^{\,top}_f)$ should be a subset of $$\bigcap_{m>0} (m V(b_{f^m}))\cap V(b_f).$$

\end{rem}

\section{$\cD$-modules}

In this section we prove Theorem \ref{thmConjIn} and Propositions  \ref{PropImp} and \ref{PropCondInd}.

\subsection{Explicit Riemann-Hilbert correspondence.}\label{subsERH}
Let $\cD_X$ be the sheaf of holomorphic differential operators on the complex manifold $X=\bC^n$. Let $F=(f_1,\ldots ,f_r)$ with $0\ne f_j\in\bC[x_1,\ldots ,x_r]$, let 
$$
f=\prod_{j=1}^rf_j,
$$
$D=f^{-1}(0)$, and let $j:U=X\smallsetminus D\ra X$ be the natural inclusion.  For  $\al$ in $\bC^r$, consider the left  $\cD_X$-modules
\begin{align*}
\cM_\al &= \cO_X[f^{-1}]\prod_{j=1}^rf_j^{\al_j},\\
\cP_\al &= \cD_X\prod_{j=1}^rf_j^{\al_j}.
\end{align*}
On $U$, $\cM_\al$ and $\cP_\al$ define the same rank one locally free $\cO_U$-module, hence an integrable connection. Let  $\cL_{\al}$ be the associated rank one local system on $U$. Let  $IC(\cL_{\al}[n])$ be the intersection complex on $X$ of the perverse sheaf $\cL_{\al}[n]$, and let $DR$ denote the De Rham  functor. The following is known to experts:

\begin{thm}\label{thmDR} (a) $\cM_\al$ is a regular holonomic $\cD_X$-module, and
$$DR(\cM_\al)=DR(\cP_{\al-m\cdot\bone})=Rj_*\cL_{\al}[n]$$
for integers $m\gg 0$.

(b) For integers $m\gg 0$, $\cP_{\al+m\cdot\bone}$ is a regular holonomic simple $\cD_X$-module, and
$$DR(\cP_{\al+m\cdot\bone})=IC(\cL_{\al}[n]).$$
\end{thm}
\begin{proof}
That
$
\cM_\al = \cD_xf^{-m}\prod_{j=1}^rf_j^{\al_j}=\cP_{\al-m\cdot\bone},
$
for integers $m\gg 0$, it is proved in \cite[Proposition 3.5]{OTa}. In {\it loc. cit.}, $\cM_\al$ is called a Deligne module and it is identified with the sheaf of sections of moderate growth of $j_*(\cO_U\otimes_{\bC_U}\cL_{\al})$, For this identification, see in \cite{Bj}: 6.1.10 and remark after 6.3.13. For the fact that the Deligne module $\cM_\al$ is a regular holonomic  $\cD_X$-module, see \cite[5.3.8]{Bj}. The unique object of $D^b_{r.h.}(\cD_X)$ such that its image under  DR  is $Rj_*\cL_{\al}[n]$ is called $\cB_+(U,\cL_{\al})$ in \cite[5.5.5]{Bj}. Note that in {\it loc. cit.} the definition of perverse sheaves is shifted by $[-n]$ compared with the commonly accepted definition of today. The identification of $\cB_+(U,\cL_{\al})$ with $\cM_\al$ is in two steps. Firstly, $\cB_+(U,\cL_{\al})$ is the $\cD$-module direct image, under a log resolution $Y$ of $(X,D)$ keeping $U$ intact, of the Deligne module attached to $U$ and $\cL_{\al}$ on $Y$, see \cite[5.5.28]{Bj}. Secondly, this direct image is the Deligne module attached to $U$ and $\cL_{\al}$ on $X$, see the proof of \cite[4.5.2]{Bj}. This shows (a).

The minimal Deligne extension $\cM_{\al,\otimes}$ of the Deligne module $\cM_\al$ is defined in \cite[4.4.8]{Bj} and also in \cite[4.4.3]{Gi}. The fact that it is regular holonomic is proved in \cite[Sublemma 2, p. 211]{Bj} and also in \cite[4.4.5.6]{Gi}. The fact that $\cM_{\al,\otimes}=\cP_{\al+m\cdot\bone}$ for integers $m\gg 0$ is proved in \cite[4.4.7]{Gi}. The fact that $DR(\cM_{\al,\otimes})=IC(\cL_{\al}[n])$ follows from the functoriality in the Riemann-Hilbert correspondence, see \cite[5.5.11]{Bj}, or see the proof of the main properties characterizing the intersection complex in \cite[4.4.4]{Gi}. This shows (b).
\end{proof}

For a point $x$ in $X$, let $U_x$ denote a small ball around $x$ in $X$, and $U_{F,x}$ denote, as before, the complement  of $D$ in $U_x$.

\subsection{Proof of Theorem \ref{thmConjIn}.}\label{subsPTM} By Proposition \ref{propSuppSp}, Lemma \ref{lemRedSp}, and Proposition \ref{propBres}, it is enough to restrict to the case when the $f_j$ with $f_j(x)=0$ define mutually distinct reduced and irreducible hypersurface germs at $x$. By Theorem \ref{thm2}, it is then enough to prove that 
$$
\Exp(V(B_{F,x}))\supset \bigcup_{y\in D\text{ near }x}\cV^{unif}(U_{F,y}).
$$
Since 
$$V(B_{F,x})=\bigcup_{y\in D\text{ near }x}V(B_{F,y}),$$
It is enough to show that 
$$
\Exp(V(B_{F,x}))\supset \cV^{unif}(U_{F,x}).
$$
When the $f_j$ vanishing at $x$ are mutually distinct reduced and irreducible hypersurface germs at $x$, the restriction of the local system defined above, $\cL_{\al}$,  to $U_{F,x}$  is the local system with monodromy $\Exp(\al_j)$ around $f_j$.

Let $\al\in\bC^r$ such that $\Exp(\al)$ is not in $\Exp(V(B_{F,x}))$. It is then enough to show that $\cL_{\al}$ is not in $\cV(U_{F,x})$. Since $H^k(U_{F,x},\cL_{\al})^\vee=H^{n-k}(U_{F,x},\cL_{\al})$, it is enough to show $\cL_{\al}$ has trivial cohomology on $U_{F,x}$. Since $\al+m\cdot\bone$ is not in $V(B_{F,x})$ for any integer $m$, $$\cP_{\al}=\cP_{\al+m\cdot\bone}$$ on $U_x$ for all integers $m$, see \cite[Proposition 3.3]{OTa}. Thus, by Theorem \ref{thmDR}, on $U_x$ we have  $Rj_*\cL_{\al}=IC(\cL_{\al})$. In particular $Rj_*\cL_{\al}=j_*\cL_{\al}$, and so $R^kj_*\cL_{\al}=0$ for $k>0$. Since $U_x$ is very small, this means that $$
 H^k(U_{F,x},\cL_\al)=0
$$
for $k>0$. When $k=0$, 
$H^0(U_{F,x},\cL_\al)\ne 0$ if and only if $\cL_\al$ is the trivial local system on $U_{F,x}$. However, we have excluded this case since $V(B_{F,x})$ contains the hyperplanes  $V(s_j+1)$ for all $1\le j\le r$ with $f_j(x)=0$. Indeed, $\id{s_j+1}=B_{F,y}$ for $y$ near $x$ such that $y$ is a nonsingular point of $D$ with $f_j(y)=0$. $\Box$

\subsection{Proof of Proposition \ref{PropImp}.} If Conjecture \ref{conj2} is true, then, by Theorem \ref{ThmGenLib}, $\Exp(V(B_{F,x}))$ is a finite union of torsion translated subtori of $S^*$. Since a subtorus of $S^*$ must be the $\Exp$ image of a  linear subvariety defined over $\bQ$ of $S$, it follows that $V(B_{F,x})$ is a union of linear subvarieties defined over $\bQ$. Suppose 
$$V(B_{F,x})\supset V(\al_1s_1+\ldots+\al_rs_r+\al)$$ 
for some $\al_j ,\al$ in $\bQ$. By \cite{G}, we know that $V(B_{F,x})$ is included in a finite union
$$
\bigcup_i V(\beta_{i1}s_1+\ldots \beta_{ir}s_r+\beta_i)
$$
with $\beta_{ij}\in \bQ_{\ge 0}$ and $\beta_i\in\bQ_{>0}$. If $\al_j\ne 0$, by restricting to $V(\id{s_k\mid k\ne j})$, we obtain that
$$
\bigcup_i\{-\beta_i/\beta_{ij}\mid \beta_{ij}\ne 0\} \supset \{-\al/\al_j\}.
$$
In particular, $\al\ne 0$ and has the same sign as $\al_j$. Thus we can assume $\al_j\in\bQ_{\ge 0}$ for all $j$ and $\al\in\bQ_{>0}$. Indeed, the only other case that can occur is when all nonzero $\al_j$ and $\al$ are negative, in which case we can replace them with their absolute values. This proves the claim. $\Box$

\subsection{The converse of Theorem \ref{thmConjIn}.}

While the converse of Theorem \ref{thmConjIn} can be phrased in terms of $\cD$-modules using Sabbah specialization for $\cD$-modules, see  \cite[\S 4 and \S 5]{Sab}, we opt for a different strategy which has the advantage of being easier to state. 

Let 
\begin{align*}
\cN & := \cD_X[s_1,\ldots,s_r]f_1^{s_1}\ldots f_r^{s_r}
\end{align*}
as a $\cD_X$-submodule of $\cO_X[f^{-1},s_1,\ldots,s_r]f_1^{s_1}\ldots f_r^{s_r}.$
There is an injective map
\begin{align*}
\nabla &: \cN \lra \cN 
\end{align*}
which sends $(s_1,\ldots,s_r)$ to $(s_1+1,\ldots, s_r+1)$. Restating Proposition \ref{corDI} locally, for a point $x\in D$, one has the following $\cD$-module description on Bernstein-Sato ideals:

\begin{lem}\label{lemBRef} 
$$B_{F,x}=Ann_{\bC[s_1,\ldots,s_r]} (\cN_x/ \nabla\cN_x).$$ \end{lem}

\begin{lem}\label{lemAssu} The following are equivalent:
\begin{enumerate}
\item For all $\al\in V(B_{F,x})$, 
$$
\sum_{j=1}^r(s_j-\al_j) \cN_x/\nabla\cN_x \subsetneq \cN_x/\nabla\cN_x;
$$
\item For all maximal ideals $I$ of $\bC[s_1,\ldots,s_r]$ containing $Ann_{\bC[s_1,\ldots,s_r]} (\cN_x/ \nabla\cN_x)$, 
$$
I(\cN_x/ \nabla\cN_x)\ne \cN_x/ \nabla\cN_x.
$$
\end{enumerate}
\end{lem}

\begin{rem} When $r=1$, the equivalent statements in the previous lemma hold. In this case, $\cN/\nabla\cN$ is a holonomic $\cD_X$-module, hence artinian. Thus the map $s-\al$ on $\cN/\nabla\cN$ is surjective if and only if it is an isomorphism, or, in other words, if and only if $\al$ is not a root of the classical one-variable Bernstein-Sato polynomial of $f$. But here is another proof of this case, without using any $\cD$-module theory. Consider the following statement:

\medskip

{\it (*) Let $M$ be a module over the localization $R$ of $\bC[s_1,\ldots,s_r]$ at the origin. Let $I$ be the maximal ideal of $R$. Assume that  $0\ne Ann_RM\ne R$. Then $I M\ne M$.}

\medskip
\noindent When $M$ is finitely generated, (*) is true due to Nakayama's Lemma. When $r=1$, (*) is again true. Indeed, in this case $Ann_RM=I^a$, for some $a>0$. If $I M=M$, then $M=I M=I^2M=\ldots = I^a M=0$, contradicting that $Ann_RM\ne R$. Now, after a linear change of the coordinates $s_j$ and after localization, (*) is equivalent  with the statement (2) of the above lemma. Hence, we have given another, simpler, proof that the statements in the previous lemma hold in the case $r=1$.

\end{rem}

\begin{rem} One way to see where lies the difficulty with proving the equivalent statements in the previous lemma for the case $r>1$ is to see why the statement (*) fails in general for non-finitely generated modules $M$. With the notation as in (*), let $R$ be the localization at the origin of the affine coordinate ring of a generic line through the origin in $\bA^r$. One can reduce to the case $r=1$ if the assumptions of (*) imply
$$Ann_{R'}(M\otimes_RR')= Ann_R(M)\otimes_RR'.$$
The left-hand side always includes the right-hand side, but the other inclusion is not always true.
\end{rem}

\subsection{Proof of Proposition \ref{PropCondInd}.} By Proposition \ref{propSuppSp}, Lemma \ref{lemRedSp}, and Proposition \ref{propBres}, it is enough to restrict to prove the converse of Theorem \ref{thmConjIn} for the case when the $f_j$ with $f_j(x)=0$ define mutually distinct reduced and irreducible hypersurface germs at $x$. By Theorem \ref{thm2}, it is then enough to prove that 
$$
\Exp(V(B_{F,x}))\subset \bigcup_{y\in D\text{ near }x}\cV^{unif}(U_{F,y}).
$$

 We follow the strategy as in the case $r=1$ from \cite[6.3.5]{Bj}. Let $\al\in V(B_{F,x})$.  Suppose that $\Exp(\al)\not\in\cV^{unif}(U_{F,y})$ for all $y\in D$ near $x$. Then $H^k(U_{F,y},\cL_{\al})=0$ for all $k$ and all $y\in D$ near $x$, where $\cL_\al$ is as in \ref{subsERH}. In particular, $R^kj_*\cL_\al$ has no sections over a small ball around such $y$, $U_y$, included in a small ball around $x$, $U_x$. Hence $$(Rj_*\cL_\al)_{| U_x\cap D}=0,$$
and so $Rj_*\cL_\al= IC(\cL_\al)$ on $U_x$. Hence $\cM_{\al}=\cP_{\al}$. In particular, $\cP_\al=\cP_{\al+m\cdot\bone}$ for any integer $m$. We will show that this contradicts the assumption.

Let 
$$
\cN_\al = \cN / \sum_{j=1}^r(s_j-\al_j)\cN.
$$
Then $\nabla$ induces a map 
$$
\rho_\al:\cN_{\al+\bone}\lra \cN_\al.
$$
The assumption is equivalent to the statements of Lemma \ref{lemAssu}. Hence  locally at $x$
$$
\sum_{j=1}^r(s_j-\al_j) \cN/\nabla\cN \subsetneq \cN/\nabla\cN,
$$
and thus the map $\rho_{\al}$ is not surjective on the stalks at $x$.  There is a natural  commutative diagram of $\cD_X$-modules
$$
\xymatrix{
\cN_{\al+\bone}   \ar[r]^{\rho_\al}\ar@{>>}[d]  &  \cN_\al \ar@{>>}[d]\\
\cP_{\al+\bone} \ar@{=}[r] &  \cP_\al}
$$
where the vertical maps are surjective. Replace, if necessary, $\al$ by $\al-m\cdot \bone$ for some positive integer $m$ to obtain that $\al\in V(B_{F,x})$, but
$$
\al-m\cdot \bone\not\in V(B_{F,x}),\quad\text{ for all }m\in\bZ_{>0}.
$$
It is possible to do so by \cite[Proposition 3.2]{OTa}. Note that $\cM_\al$ is unchanged. Then, by the local version of \cite[Proposition 3.6]{OTa}, the right-most map gives an isomorphism locally at $x$
$$
\cN_\al\cong \cD_{X}f_1^{\al_1}\ldots f_r^{\al_r}.
$$
Since $\rho_{\al}$ is not surjective locally at $x$, it follows that 
$$
\cD_{X}f_1^{\al_1+1}\ldots f_r^{\al_r+1}\subsetneq \cD_{X}f_1^{\al_1}\ldots f_r^{\al_r},
$$ which is what was claimed.
$\Box$

\section{Hyperplane arrangements}\label{subsHA}

In this section we give a combinatorial formula for the support of the Sabbah specialization complex of a collection of hyperplanes, prove Corollaries \ref{corBFHA} and \ref{corCO}, give a different proof of the fact the Multi-Variable Monodromy Conjecture holds for hyperplane arrangements, and prove Theorem \ref{thmSM}.

\subsection{Terminology.}\label{subTer}  Let 
$D=\cup_{j=1}^rD_j$ be a finite collection of hyperplanes in $X=\bC^n$. Let $f_j$ be a linear polynomial defining $D_j$. Let $f=\prod_jf_j^{m_j}$ with $m_j\ge 1$ be a possibly non-reduced polynomial such that the zero locus $V(f)$ is $D$. We call both $D$ and $f$ {\it hyperplane arrangements}. The 
 hyperplane arrangement is called \emph{central} if each
$D_j$ is a linear subspace of codimension one of $X$, that is, if $f$ is homogeneous.
 A central hyperplane arrangement $f$ is {\it
indecomposable} if  there is no linear change of coordinates
on $X$ such that $f$ can be written as the product of two
non-constant polynomials in disjoint sets of variables. Note that indecomposability is a property of the underlying reduced  zero locus $D$ of $f$.
An {\it edge} is
an intersection of hyperplanes $D_j$.
An arrangement is \emph{essential} if $\{0\}$ is an edge.
For a linear subset $W\subset X$, the {\it restriction of $D$ to $W$} is the hyperplane arrangement $D_{W}$ given by
the image of $\cup_{D_j\supset W}D_j$ in the vector space quotient $X/W$, which is defined as soon as we make a linear change of coordinates such that $0\in W$.  Similarly, one defines {\it the restriction $f_W$ of $f$ to $W$} as a polynomial map on $X/W$, by keeping track of the multiplicities along the hyperplanes $D_j$ which contain $W$. An edge $W$ of $D$ is
called {\it dense} if the restriction arrangement ${D}_{W}$  is
indecomposable. For example, $D_j$ is a dense edge for every $j$.   The {\it canonical log resolution} of $D$ is the map
$\mu:Y\ra\bC^n$ obtained by composition of, in increasing order for $i=0,1,\ldots,n-2$, the blowups along
the (proper transform of) the union of the dense
edges of dimension $i$. 

\begin{thm}\label{propDP} ( \cite[Theorem 3.1]{STV})
The canonical log resolution $\mu:Y\ra X$ is a log resolution of $f$.
\end{thm}

\begin{prop}\label{lemIndecChi} (\cite[Proposition 2.6]{STV}) If $f$ is a central hyperplane arrangement and  $V=\bP^{n-1}- \bP (D)$, then
$f$ is indecomposable if and only if $\chi (V)\ne 0$.
\end{prop}

\begin{rem}
The Euler characteristic of the complement of a hyperplane arrangement can be determined only from the lattice of intersections of the hyperplanes in the arrangement, see \cite{OTe}. Hence the previous Proposition also implies that indecomposability and density are combinatorial conditions.
\end{rem}

\subsection{Sabbah specialization complex for arrangements}
From now on we use the same setup as in \ref{subsSSP}. Assume that $f_j$ are central hyperplane arrangements in $X=\bC^n$, not necessarily reduced, of degree $d_j$.

The following two lemmas are immediate consequences of Corollary \ref{corZH}, Proposition \ref{propSuppHom}, and Proposition \ref{lemIndecChi}:

\begin{lem}\label{lemIndecMondromy} If $f=\prod_{j=1}^rf_j$ is an indecomposable central hyperplane arrangement, then 
$V(t_1^{d_1}\cdots t_r^{d_r}-1)\subset\Supp _0(\psi_{F}\bC_X).$ 
\end{lem}

\begin{lem}\label{lemAllSupp} There is an equality in Lemma \ref{lemIndecMondromy}, if, in addition, $f$ is reduced.\end{lem}

\begin{df} We say that polynomial map $F=(f_1,\ldots ,f_r)$ with $f_j\in\bC[x_1,\ldots,x_n]$ {\it splits into} $G\cdot H$, and that $G\cdot H$ is a {\it splitting} of $F$, if, up to a different choice of coordinates, there exists $m$ with $1\le m\le n$ and there are polynomials $g_j(x_1,\ldots ,x_m)$ and $h_j(x_{m+1},\ldots,x_n)$ for $j=1,\ldots, r$, such that not all $g_j$ are constant, not all $h_j$ are constant, and $f_j=g_jh_j$. If so, we set $G=(g_1,\ldots,g_r)$ and  $H=(h_1,\ldots ,h_r)$. Otherwise, we say that $F$ is {\it does not split}. We say that a splitting
$$
F=F^{(1)}\cdot\ldots\cdot F^{(l)}
$$
is {\it total} if each $F^{(i)}$ does not split. 
\end{df}

Let $F=(f_1,\ldots, f_r)$ be such that $f_j$ are linear forms defining mutually distinct hyperplanes. Up to multiplication by constants, a total splitting of $F$ is unique. For an edge $W$ of the hyperplane arrangement $f=\prod_{j=1}^rf_j$, let 
$$
F_W=(f_{1,W},\ldots ,f_{r,W}) : X/W \lra S^*,
$$
where $f_{j,W}$ is the restriction of the hyperplane arrangement $f_j$ to $W$ as defined in \ref{subTer}. More precisely, $f_{j,W}={f_j}_{| X/W}$ if $f_j(W)=0$, and $f_{j,W}=1$ otherwise. Note that $W$ is a dense edge if and only if $F_W$ does not split. For every edge, let 
$$
F_W=\prod_{i=1}^{l_W}F_W^{(i)}
$$
be a total splitting of $F_W$. If we set $F_W^{(i)}=(f_{1,W}^{(i)},\ldots ,f_{r,W}^{(i)})$, let
$$
d_{j,W}^{(i)}:=\deg f_{j,W}^{(i)}.
$$
Note that $d_{j,W}^{(i)}$ is either 0 or 1.

\begin{prop}\label{propBFHA} (a) If  $f_j$ are linear forms defining mutually distinct hyperplanes, then 
\begin{equation}\label{eqFHA}
\bigcup_{x\in D}\Supp ^{unif}_x (\psi_F\bC_X) = \bigcup_{W} V(\id{t_1^{d_{1,W}^{(i)}}\ldots t_r^{d_{r,W}^{(i)}}-1\mid i=1, \ldots ,l_W}),
\end{equation}
where the union is over the edges $W$ of the hyperplane arrangement $\prod_{j=1}^rf_j$. In particular, the codimension-one part is the zero locus in $S^*$ of 
$$
\prod_{W} \left (  \prod_{j\ :\ f_j(W)=0 } t_j -1  \right ),
$$
where the first product is over dense edges $W$ of $f=\prod_jf_j$. 

(b) Assuming Conjecture \ref{conj2}, formula (\ref{eqFHA}) also holds for 
$\Exp(V(B_F))$. 
\end{prop}
\begin{proof} First, we assume that $f=\prod_{j=1}^rf_j$ is a central hyperplane arrangement. Then
$$
\bigcup_{x\in D}\Supp ^{unif}_x (\psi_F\bC_X) =\Supp ^{unif}_0 (\psi_F\bC_X) \cup \bigcup_{0\ne y\in D}\Supp ^{unif}_y (\psi_F\bC_X).
$$

Let us focus on the first term of the right-hand side. Let 
$$
F_0=F=F_0^{(1)}\cdot\ldots\cdot F_0^{(l_0)}
$$
be a total splitting of $F_0$, with $F_0^{(i)}$ defined on $X_i$, and $X=\times_{i=1}^{l_0}X_i$. By Proposition \ref{propMTS},  we have
$$
\Supp ^{unif}_0 (\psi_F\bC_X) = \bigcap_{i=1}^{l_0} \Supp ^{unif}_0 (\psi_{F_0^{(i)}}\bC_{X_i}).
$$
Since $F_0^{(i)}$ is does not split, the hyperplane arrangement $\prod_{j=1}^rf_{j,0}^{(i)}$ is indecomposable. Hence, by Proposition \ref{lemIndecChi}, for each $F_0^{(i)}$ we are in the case Lemma \ref{lemAllSupp}. Thus,
$$
\Supp _0 (\psi_{F_0^{(i)}}\bC_{X_i}) =  V(t_1^{d_{1,0}^{(1)}}\ldots t_{r}^{d_{r,0}^{(l_0)}}-1)
$$
inside the torus
$$
\Spec(\bC[t_1^{\pm d_{1,0}^{(1)}},\ldots ,t_r^{\pm d_{r,0}^{(l_0)}}]).
$$
By the definition of the uniform support, $\Supp _0 (\psi_{F_0^{(i)}}\bC_{X_i})$ has the same equations, but inside the possibly-bigger torus
$$
\Spec(\bC[t_1^{\pm 1},\ldots ,t_r^{\pm 1}]) = S^*.
$$
Thus,
$$
\Supp ^{unif}_0 (\psi_F\bC_X)   =V(\id{t_1^{d_{1,0}^{(1)}}\ldots t_{r}^{d_{r,0}^{(l_0)}}-1\mid i=1,\ldots , l_0}).
$$
The rest of the claim follows by replacing $W=0$ in the above argument with other edges $W$ of the hyperplane arrangement $\prod_{j=1}^rf_j$.

If $f$ is not central, fix $x\in D$. The above argument for the central case gives the equations of the support of $\psi_F\bC_X$ at $x$ inside the torus 
$$
\bT_x=V(\id{t_j-1\mid f_j(x)\ne 0})\subset S^*.
$$
By the definition of the uniform support, $\Supp ^{unif}_x (\psi_F\bC_X)$ has the same equations in $S^*$, and the claim follows.
\end{proof}

\begin{rem}\label{remRel}
If $F=(f_1,\ldots ,f_r)$ is such that $\prod_{j=1}^rf_j$ is a possibly non-reduced hyperplane arrangement, then $F$ is the specialization of a polynomial map satisfying the conditions of Proposition \ref{propBFHA} up to the harmless appearance of additional constant polynomials, as explicited in Lemma \ref{lemRedSp}. Since we know how the supports behave under specialization, there is no mystery then what happens in Proposition \ref{propBFHA} after dropping the conditions.
\end{rem}

\subsection{Proof of Corollary \ref{corBFHA}.}
This is an immediate consequence of Proposition \ref{propBFHA} and Theorem \ref{thmConjIn}. $\Box$

\subsection{Proof of Corollary \ref{corCO}.} 
We can assume that $f$ is reduced. The general case will follow from this case in light of Remark \ref{remRel}. First, we need to clarify the notation used in the statement. With the notation as in Proposition \ref{propBFHA}, for an edge $W$ of $D$ let $f_W$ be the restriction of the hyperplane arrangement $f$ to $W$ as in \ref{subTer}. If
$$
f_W^{(i)}:=\prod_{j=1}^r f_{j,W}^{(i)},
$$
then
$$
f_W=\prod_{i=1}^{l_W}f_W^{(i)}
$$
is a total splitting of $f_W$. Let $$d_W^{(i)}:=\deg f_W^{(i)}=\sum_{j=1}^rd_{j,W}^{(i)}.$$

Let $f_j$ for $j=1,\ldots, r$, be the irreducible factors of $f$. Specialize $F=(f_1,\ldots, f_r)$ to $f=\prod_{j=1}^rf_j$ as in Example \ref{exaSpG}. As in the proof of Theorem \ref{thmSp}, 
$$\bigcup_{x\in D}\Supp ^{unif}_x (\psi_F\bC_X)$$
specializes via the restriction to the diagonal ${t_1=\ldots =t_r}$ to the set 
$$
\star =\bigcup_{x\in D}\bigcup_k\{\text{eigenvalues of monodromy on }H^{k}(M_{f,x},\bC) \},
$$
where $M_{f,x}$ is the Milnor fiber of $f$ at $x$, see Remark \ref{remOV}. Since
$$d_W^{(i)}=\sum_{j=1}^rd_{j,W}^{(i)},$$
we have
$$
\{t_1=\ldots =t_r\}\cap \bigcup_{x\in D}\Supp ^{unif}_x (\psi_F\bC_X)=\bigcup_WV(\id{t^{d_W^{(i)}}-1 \mid i=1,\ldots, l_W}),
$$ 
by Proposition \ref{propBFHA}. The conclusion follows from the fact that $\Exp( V(b_f))=\star$, by the theorem of Malgrange and Kashiwara. $\quad\Box$

\begin{thm}\label{thmMainThm} If each $f_j$ define a (possibly nonreduced) hyperplane arrangement in $\bC^n$, then Conjecture \ref{conjMVMonConj} holds.
\end{thm}
\begin{proof} We give another proof of this result, different than the one given by Theorem \ref{thmR1}. We deal first with the central case. For $j=1,\ldots , r$, let $f_j$ be central hyperplane arrangements in $X=\bC^n$. Using the canonical log resolution, we see that the polar locus
$
PL(Z_{F}^{\, top})
$
is a hyperplane sub-arrangement of  $\cup_WP_W$, with
$$P_W= \{a_{W,1}s_1+\ldots +a_{W,r}s_r+k_W+1=0\},$$ where $W$ varies over the dense edges of the hyperplane arrangement $D$, $a_{W,j}=\ord _{W}(f_j)$, and $k_W=\codim (W) -1$.

Fix a dense edge $W$, and the corresponding hyperplane $P_W$ which candidates for a component of the polar locus. Let $D_W$, $f_W$, $f_{j,W}$ be the restrictions of the hyperplane arrangements  $D$, $f$,  $f_j$, respectively, to $W$ as defined in \ref{subTer}.  We have $f_W=\prod_j f_{j,W}$, where the product is over those $j$ with $f_j(W)=0$. We can assume $\{j \mid f_j(W)=0 \}=\{1,\ldots , p\}$ for some integer $p$.
Now,
take a point $x\in W$ not lying on any hyperplane in $D$ which does not contain $W$. After choosing a splitting of $W\subset \bC^n$, we have locally around $x$,
${D}={D}_W\times W\subset \bC^n=\bC^n/W\times W$ and
$f=f_W\cdot u$, where $u$ is a (locally) invertible function. Hence, by Lemma \ref{lemIndecMondromy}, 
$$V(\prod_{j =1}^{p}t_j^{a_{W,j}}-1)\subset \Supp _x(\psi_{f_{W,1},\ldots ,f_{W,p}}\bC_{\bC^n/W})\subset (\bC^*)^{p}.$$
On the other hand, by  the definition of the uniform support, we have 
$$V(\prod_{j =1}^{p}t_j^{a_{W,j}}-1)\subset\Supp _x^{unif}(\psi_{F}\bC_X)\subset (\bC^*)^{p}\times (\bC^*)^{r-p}=S^*.$$
Let $\lam \in P_W$. Then
$$
\prod_{j=1}^r \left (e^{2\pi i \lam_j}\right )^{a_{W,j}} -1= e^{-2\pi i \cdot  codim(W)} -1 = 0.
$$
This shows that 
$$
\Exp ( P_W)\subset V(t_1^{a_{W,1}}\cdots t_r^{a_{W,r}} -1)\subset \Supp _x^{unif}(\psi_{f_1,\ldots ,f_r}\bC_X),$$
which was the claim.

The non-central case follows as in the proof of Proposition  \ref{propBFHA}. \end{proof}

\subsection{Proof of Theorem \ref{thmSM}.} 
We prove the claim for the case when $f=\prod_{j=1}^rf_j$ is central, since this implies the non-central case as well.  As in the proof of Theorem \ref{thmMainThm}, the polar locus  $
PL(Z_F^{\, top})
$
is included in the hyperplane arrangement  $\cup_WP_W$, with
$$P_W= \{a_{W,1}s_1+\ldots +a_{W,p}s_{p}+k_W+1=0\},$$ where $W$ varies over the dense edges of the hyperplane arrangement $D$, $a_{W,j}=\ord _{W}(f_j)$, and $k_W=\codim (W) -1$. Hence,  
$$P_W=\{\deg (f_{1,W})s_1+\ldots \deg (f_{p,W})s_{p}+\dim (\bC^n/W) =0 \}.$$
Note that $f_W=\prod_{j=1}^{p}f_{W,j}$ is indecomposable, and automatically central and essential. By assumption, $P_W$ is in the zero locus of the ideal $B_{F_{W}}$, where $F_W=(f_{1,W},\ldots ,f_{p,W})$ as before.

Now, as in the proof of Theorem \ref{thmMainThm}, take a point $x\in W$ not lying on any hyperplane in $D$ which does not contain $W$. We have $B_{F_{W}}=B_{F,x}$, and the zero locus of $B_{F,x}$ is included in the zero locus of $B_F$. The claim follows.
$\Box$

\section{Examples}\label{secEx}

\begin{exa}\label{exaIntr} Let $F=(x,y,x+y,z,x+y+z)$. Then the product of all entries of $F$ forms a central essential indecomposable hyperplane arrangement in $\bC^3$. The Bernstein-Sato ideal $B_{F}$  of $F$ is currently intractable via computer. However, Conjecture \ref{conj2}  predicts via Corollary \ref{corBFHA} that, in $(\bC^*)^5$,
\begin{equation}\label{eqE1}
\Exp( V(B_F))=V(\id{(t_1t_2t_3-1)(t_3t_4t_5-1)(t_1\ldots t_5-1)\prod_{j=1}^5(t_j-1)}).
\end{equation}
We can actually check this as follows. One can compute with {\tt dmod.lib} \cite{SingDmod}:
\begin{align*}
B_{F}^{\be_1} &=\id{(s_1+1)(s_1+s_2+s_3+2)(s_1+s_2+s_3+s_4+s_5+3)},\\
B_{F}^{\be_2} &=\id{(s_2+1)(s_1+s_2+s_3+2)(s_1+s_2+s_3+s_4+s_5+3)},\\
B_{F}^{\be_3} &=\id{(s_3+1)(s_1+s_2+s_3+2)(s_3+s_4+s_5+2)(s_1+s_2+s_3+s_4+s_5+3)},\\
B_{F}^{\be_4} &=\id{(s_4+1)(s_3+s_4+s_5+2)(s_1+s_2+s_3+s_4+s_5+3)},\\
B_{F}^{\be_5} &=\id{(s_5+1)(s_3+s_4+s_5+2)(s_1+s_2+s_3+s_4+s_5+3)}.
\end{align*}
Then (\ref{eqE1}) follows from Lemma \ref{lemStr} which says that $\Exp(V(B_F))=\bigcup_{j=1}^5\Exp(V(B_F^{\be_j}))$. Using a different computation, U. Walther has also checked that (\ref{eqE1}) holds. 

The local topological zeta function $Z^{\; top}_{F,0}$ of $F$ at the origin has a degree-7 irreducible numerator, and the denominator is equal to 
$$
 (s_1+s_2+s_3+2)(s_3+s_4+s_5+2)(s_1+\ldots +s_5+3)\prod_{j=1}^5(s_j+1).
$$
This illustrates the (local version of the) Multi-Variable Monodromy Conjecture which we proved for hyperplane arrangements. This also shows that Conjecture \ref{conjND} and the (local version of the) Multi-Variable Strong Monodromy Conjecture  hold for $F$.
\end{exa}

\begin{exa} Conjecture \ref{conjND}, which we proved to imply the Multi-Variable Strong Monodromy Conjecture for hyperplane arrangements, can fail for decomposable arrangements. Let $F=(x,y,x+y,z)$. Then the product of all entries of $F$ forms a central essential, but decomposable, hyperplane arrangement. The hyperplane $s_1+\ldots +s_4+3=0$ does not lie in the zero locus of the Bernstein-Sato ideal
$$B_{F}=\id{
(s_1+1)(s_2+1)(s_3+1)(s_4+1)\prod_{k=2}^4(s_1+s_2+s_3+k)}.
$$
\end{exa}

\begin{exa} If $F$ is as in Example \ref{exaNonP}, the multi-variable local topological zeta function at the origin is
\begin{align*}
Z^{\;top}_{F,0}(s)=\frac{2s_1s_2+s_2+1}{(s_1+1)(s_2+1)(2s_2+1)}.
\end{align*}
This illustrates the (local version of the) Multi-Variable Strong Monodromy Conjecture.
\end{exa}

\end{document}